\documentclass[draft]{amsart}

\usepackage{amssymb,latexsym,amscd, amsthm, epsfig,amsmath,amssymb}
\usepackage{color}
\usepackage{array}
\usepackage{enumerate}
\usepackage{enumitem}
\usepackage{refcount}
\usepackage{algorithm}
\usepackage{algorithmic}
\usepackage{boxedminipage}
\usepackage{float}
\usepackage[utf8x]{inputenc}
\newcolumntype{C}{>{$}c<{$}}
\newcolumntype{L}{>{$}l<{$}}
\newcolumntype{R}{>{$}r<{$}}

 \newtheorem{thm}{Theorem}[section]
 \newtheorem{cor}[thm]{Corollary}
 \newtheorem{lem}[thm]{Lemma}
 \newtheorem{prop}[thm]{Proposition}
 \theoremstyle{definition}
 \newtheorem{defn}[thm]{Definition}
 \newtheorem{notation}{Notation}
 \newtheorem{rem}[thm]{Remark}
 \theoremstyle{definition}
 \newtheorem{ex}[thm]{Example}
 \newtheorem{conjecture}[thm]{Conjecture}

\newtheorem{prob*}{Problem}

 \DeclareMathOperator{\Supp}{Supp} 
 \DeclareMathOperator{\Sing}{Sing} 
 \DeclareMathOperator{\length}{length}

 \newcommand{\RR}{\mathbb{R}}
 
 \newcommand{\CC}{\mathbb{C}}

 \newcommand{\PP}{\mathbb{P}}

\def\move-in{\parshape=1.75true in 5true in}

\newcommand\algorithmicinput{\textbf{Input:}}
\newcommand\INPUT{\item[\algorithmicinput]}
\newcommand\algorithmicoutput{\textbf{Output:}}
\newcommand\OUTPUT{\item[\algorithmicoutput]}

\usepackage{xy}

\usepackage{todonotes}

\begin{document}

\title{Singularities of plane rational curves via projections}

\author[A. Bernardi]{Alessandra Bernardi}
\address[Alessandra Bernardi]{Dipartimento di Matematica, Univ. Trento,  Italy}
\email{alessandra.bernardi@unitn.it}

\author[A. Gimigliano]{Alessandro Gimigliano}
\address[Alessandro Gimigliano]{Dipartimento di Matematica e CIRAM,  Univ. Bologna, Italy}
\email{alessandr.gimigliano@unibo.it }

\author[M. Id\`a]{Monica Id\`a}
\address[Monica Id\`a]{Dipartimento di Matematica,  Univ. Bologna, Italy}
\email{monica.ida@unibo.it }


\begin{abstract}
We consider the parameterization ${\mathbf{f}}=(f_0:f_1:f_2)$ of  a plane rational curve $C$ of degree $n$, and we study the singularities of $C$ via such parameterization. We use the projection from the rational normal curve $C_n\subset \PP^n$ to $C$ and its interplay with the secant varieties to $C_n$. In particular, we define via ${\mathbf{f}}$ certain 0-dimensional schemes $X_k\subset \PP^k$, $2\leq k\leq (n-1)$, which encode all information on the singularities of multiplicity $\geq k$ of $C$ (e.g. using $X_2$ we can give a criterion to determine whether $C$ is a cuspidal curve or has only ordinary singularities). We give a series of algorithms which allow one to obtain information about the singularities from such schemes.  \end{abstract}


\maketitle

\section{Introduction}

The study of plane rational curves is quite classical in algebraic geometry, and it is also an interesting subject for applications, for example it is very relevant in Computer Aided Design (CAD). Since rational curves are the ones that can be parameterized, it is quite of interest to get information on the curve from its parameterization (implicit equation, structure of singularities, e.g. see \cite{refSCG, refSWP, refCSC,refCKPU}).
In this paper we tackle  the problem of determining the singularities of a plane rational curve from its parameterization. This is a problem which has been much treated in the literature: see the beautiful work \cite{refCKPU}, where the syzygies of the ideal generated by the polynomials giving the parameterization are used in order to determine the singularities of the curve 
and their structure (multiplicity, branches, infinitely near other singularities).

This idea has been developed also in \cite{refSCG} and \cite{refCWL}, where ``~$\mu$-bases~'' are exploited for the parameterized curve. We used this approach in a previous paper, \cite{BGI3}, in order to find how a plane curve could be viewed as a projection of a rational curve contained in a rational normal scroll.  

\medskip

In the present paper we describe the structure of singular points by using the parameterization, but from a different point of view with respect to the one of \cite{BGI3}. In order to study the singularities of a plane rational curve $C$ of degree $n$, we use  the fact that the parameterization of $C$ defines a projection $\pi : \PP^n  \dashrightarrow \PP^2$, which is generically one-to-one from the rational normal curve $C_n\subset \PP^n$ onto its image, and $\pi(C_n)= C\subset  \PP^2$.  If $P$ is a singular point of multiplicity $m$ of $C\subset \PP^2$, then there is an $(m-1)$-dimensional $m$-secant space $H$ to $C_n$ in $\PP^n$ such that $\pi(H)=P$. The center of projection of $\pi$ is a $(n-3)$-linear space $\Pi$, and $\Pi\cap H$ has to be $(m-2)$-dimensional, in order to have that $\pi (H)$ is a point. We have that  $H\cap C_n$ (and $H\cap \Pi$) contains all the information about the singularity $P$ of $C$ (multiplicity, branches, infinitely near points, e.g. see \cite{MOE}); the problem is how to extract this information from these data.   

Our strategy here is to consider, for $k=2,\ldots ,n-1$, the spaces $\PP^k\cong \PP (K[s,t]_k)$ that parameterize $\sigma_k(C_n)$, the $k$-secant variety of $(k-1)$-dimensional $k$-secant spaces to $C_n$ 
and their intersection with the center of projection $\Pi$, which is determined by the parameterization of $C$. Such study yields to considering certain 0-dimensional schemes, $X_k\subset \PP^k$, which parameterize the $k$-secant $(k-1)$-spaces that get contracted to a point by the projection $\pi$, so that they encode all the information on the singularities of $C$. For example, we can use the scheme $X_2$ to give simple necessary and sufficient conditions for the curve $C$ to be a cuspidal one or to have only ordinary singularities (see Proposition \ref{cuspidal}).

\medskip

This approach stems out from a study of the so-called Poncelet varieties associated to rational curves (see \cite{ISV}); in a previous paper (see \cite{BGI2}) we considered the singularities of Poncelet surfaces in order to determine the existence of triple points on $C$. Here that approach has been generalized and potentially covers all kind of singular points on $C$. The interplay between secant varieties $\sigma_k(C_n)$ of rational normal curves and 0-dimensional subschemes of the space $\PP^k$ parameterizing the $\PP^{k-1}$ $k$-secant spaces of $C_n$ has also been studied by the authors in \cite{BGI1}.  

Our choice has been to give our main results in the form of algorithms that can be used  to study a given rational curve $C$, for example  with  the help of programs as Cocoa \cite{COCOA}, Macauly2 \cite{m2} or Bertini  \cite{Bertini}. Our Algorithm \ref{trueAlg1} allows one to compute the number $N$ of singularities of $C$ 
and also the number $N_k$ of singularities of multiplicity $k$, for $k=2,\ldots ,n-1$.  We also give a variation of Algorithm \ref{trueAlg1}  (c.f. Algorithm \ref{AggiuntaAlg1}) that allows one to compute, for each multiplicity, the number of singular points with given number of branches and multiplicity of each branch. Our Algorithm \ref{Algorithm2True} computes the numbers $N,N_2,\ldots,N_{n-1}$ too, but it also gives the ideal of each subset  $\Sing_k(C)\subset \Sing(C)$ given by the points with multiplicity $k$.  Eventually, Algorithm \ref{Algorithm3True} gives the (maybe approximated) coordinates of the points in $\PP^1$ which, via the parameterization $\mathbf{f}$ of $C$, are the preimages of the singular points of $C$. This allows one  to compute  the coordinates of the singular points of $C$ by applying $\mathbf{f}$ to such points.

Although algorithms determining the structure of plane curves singularities do exist (see \cite{refCWL}, \cite{Pe}, \cite{refCKPU}), we think that our algorithms can be a useful tool, also used together with the existing ones, since their approach to the problem is different, and their behavior on specific curves can be of different effectiveness.

\medskip

The paper is organized as follows: in the next section we give all the preliminary notions and define the schemes $X_k$ which will be our crucial tool to study Sing$(C)$.  In Section 3 we give the algorithms mentioned above. Section 4 is dedicated  to the study of curves with only double points. In this case we give a criterion (Theorem  \ref{2mtom}) to describe, via the projection $\pi|_{C_n}: C_n \to C$, which kind of singularity  a double point can be
. This result is interesting in itself, since usually this description via projection uses the osculating spaces of $C_n$, but that does not work when the multiplicity of the singular point is big with respect to $n$ (e.g. see \cite[Remark 4.5.1]{MOE}).   Then we conjecture that the structure of $X_2$ allows one to recover all the information about the structure of the singular points. Following our conjecture we give an algorithm (Algorithm \ref{Algorithm4}) which may find the structure of the double points; the algorithm will always work if the conjecture is true. All along the paper we work over an algebraically closed field of characteristic zero, except in the last section where we consider $K=\mathbb{R}$, since in this case  our Algorithm \ref{Algorithm3True} can give a method to find acnodes and hidden singularities.    

\section{Preliminaries}\label{Preliminaries}

Let   $K$  an algebraically closed field with char $K$ =0. We study the singularities of a parameterized rational plane curve $C \subset \PP^2=\PP^2_K$ given by a map  ${\mathbf{f}}= (f_0:f_1:f_2)$, where $f_i \in K[s,t]_n$, $n\geq 3$. We will always suppose that our parameterization is {\it proper}, i.e. that ${\mathbf{f}}$ is generically 1:1 and the $f_i$'s, $i=1,2,3$, do not have common zeroes. Our approach will be to view $C$ as the projection of a rational normal curve $C_n\subset \PP^n$ into $\PP^2$, so that the singularities of $C$ will be related to the position of the center of projection with respect to secant (and tangent) lines and osculating spaces to $C_n$ (see e.g. \cite{MOE}); in particular we will concentrate our study mainly on the use of the secant variety $\sigma_2(C_n)$.

Let us start with studying how the varieties $\sigma_k(C_n)$ can be parameterized by a space $\PP^k$, e.g. following  the construction in \cite{ISV}.

Let  $\nu_n: \mathbb{P}^1 \rightarrow \mathbb{P}^n$ be the Veronese $n$-uple embedding and let $C_n = \nu_n(\PP ^1)\subset \PP ^n$ be the rational normal curve in $\PP ^n$. Consider the space $\PP (K[s,t]_k)\cong \PP ^k$, with $2\leq k\leq n-1$; every point in this space corresponds (modulo proportionality) to a polynomial of degree $k$,
and its roots give $k$ points (counted with multiplicity) in $\PP ^1$,
hence one of the $k$-secant $(k-1)$-planes in the  secant variety $$\sigma_k (C_n)=\overline{\bigcup_{P_1,P_2,\ldots,P_k\in C_n}\langle P_1,P_2,\ldots , P_k\rangle} \subset\PP ^n.$$

\smallskip

Notice that only for $k < \left[ \frac{n+1}{2}\right]$ we have $\sigma_k (C_n) \neq \PP^n$; for higher values of $k$ (i.e. for $\left[ \frac{n+1}{2}\right]\leq k \leq n-1$) the secant variety $\sigma_k (C_n)$ fills up $\PP^n$; nonetheless the points of $\PP^k$ still parameterize the $(k-1)$-spaces $k$-secant to $C_n$.

\medskip 

If we consider coordinates $x_0,\ldots ,x_k$ in $\PP ^n$ and coordinates $z_0,\ldots ,z_n$ in $\PP ^n$ then the variety $\sigma_k(C_n)$ can be viewed in the following way: consider $Y\subset \PP ^k\times \PP ^n$ defined by the
equations
\begin{equation}\label{eqn1}
\sum_{i=0}^k x_iz_{i+j} = 0, \; \; j=0, \ldots , n-k.
\end{equation}

We have that the $(n-k+1)\times (n+1)$ matrix of coefficients of (\ref{eqn1}) in the $z_{i+j}$'s
is:

\begin{equation}
\label{Ak}
A_k=  \begin{pmatrix} x_0&x_1&\cdots &x_k&0&0&\cdots&0 \cr
0&x_0&x_1&\cdots&x_k&0&\cdots&0 \cr 
&&\ddots&&&&& \cr
0&\cdots&0&x_0&x_1&\cdots&\cdots&x_k
\end{pmatrix}.
\end{equation}

\medskip
While the $(n-k+1)\times(k+1)$ matrix of  coefficients of (\ref{eqn1})  in the $x_i$'s is the catalecticant matrix:

\medskip
\begin{equation}\label{Bk}B_k=
\begin{pmatrix} 
z_0&\cdots&z_{k} \cr
z_1&\cdots&z_{k+1}\cr
\cdots &\cdots&\cdots
\cr z_{n-k}&\cdots &z_n
\end{pmatrix}.\end{equation}

\medskip

Then if we consider the two projections  $p_1: Y \rightarrow\PP ^k$ and
$p_2: Y \rightarrow\PP ^n$, we get that $p_1$ gives a projective bundle structure on $\PP ^k$, with fibers $\PP ^{k-1}$'s (this is known as the {\it Schwartzenberger bundle} associated to $\sigma_k(C_n)$, see \cite{Sc, ISV}).  When $k < \frac{n-1}{2}$, $p_2(Y) = \sigma_k(C_n)$ (the map $p_2$ is a desingularization of $\sigma_k(C_n)$), while when $k \geq \frac{n-1}{2}$, $\sigma_k(C_n)=\PP^n$ and $p_2$ is surjective.

 Notice that, for $k\leq \frac{n-1}{2}$, for each point $p \in \sigma _{k- i}(C_n) \backslash \sigma_{k-1-i}(C_n)$, with $i=0,1,\ldots , k-1$, the fibers of $p_2$  have dimension equal to $i$ (e.g. see  \cite{ISV}).

Moreover, for all $ P\in  \PP ^k$, we have that $ p_2(p_1^{-1}(P)) $ is a  $(k-1)$-space $k$-secant to $C_n \subset  \PP ^n$, thus showing that  $\PP ^k$ parameterizes the  $(k-1)$-secant $k$-spaces to $C_n\subset \PP^n$. Notice also that the maximal minors of $B_k$, when $k < \frac{n-1}{2}$, define the ideal of $\sigma_k(C_n)\subset \PP^n$.

In particular, when we consider the points in $\PP^k$ that parameterize $k$-osculating $(k-1)$-spaces to $C_n$ (their intersection with $C_n$ has support at one point and it has degree $k$) they are the points of the rational normal curve $\mathcal{C}_k$, that parameterizes forms in $K[s,t]_k$ which are $k$-powers of linear forms: $(as+bt)^k = \sum _{i=0}^k {k \choose i}a^{k-i}b^is^{k-i}t^i$; so when $$(x_0:x_1:\ldots : x_k)=(a^k:ka^{k-1}b:\ldots : {k \choose i}a^{k-i}b^i, \ldots : kab^{k-1}: b^k),$$ for $a,b\in \PP^1$, we get a rational normal curve $\mathcal{C}_k\subset \PP^k$.

Notice that we choose to adopt the notation $(x_0:x_1:\ldots :x_N)$, with colons, for the homogeneous coordinates of a point in $\PP^N$ (other notations, with commas or square brackets, are also common).

Here we are viewing the rational normal curves $\mathcal{C}_k\subset \PP^k$ and $C_n\subset \PP^n$ as the $k$-uple and $n$-uple Veronese embeddings of $\PP^1$ in two different ways: the curve $C_n$ is the image of the map $\nu_n:\ (s,t) \rightarrow (s^n:s^{n-1}t:\ldots :t^n)$, while $\mathcal{C}_k$ is the image of the map that sends the form $as+bt$ to $(as+bt)^k$, hence $(a,b) \rightarrow \left(a^k:\ldots : {k \choose i}a^{k-i}b^i:\ldots : b^k\right)$.  For example, $\mathcal{C}_2$ is the dual curve of the usual rational normal conic $C_2=\{z_0z_2-z_1^2=0\}$, and it has equation $4x_0x_2-x_1^2=0$ (given by the discriminant of the form $x_0s^2+x_1st+x_2t^2$, parameterized by each point $(x_0:x_1:x_2)$).

\medskip
We want to describe explicitly our curve $C\subset \PP^2$ as a projection of $C_n\subset \PP^n$; let us consider  $\langle f_0,f_1,f_2\rangle \subset
K[s,t]_n$, with $f_u = a_{u0}s^n + a_{u1}s^{n-1}t + \cdots + a_{un}t^n$, $u=0,1,2$; when we associate our coordinates
 $z_p$ with $s^{n-p}t^p$, we can associate to  $\langle f_0,f_1,f_2\rangle$ an $(n-3)$-dimensional subspace  $\Pi \subset \PP ^n$, given by the equations  
 \begin{equation}\label{fk}
  f_u({\bf{z}})=a_{u0}z_0 + a_{u1}z_{1} + \cdots + a_{un}z_n=0, \; \; u=0,1,2.
  \end{equation}

Notice that $\Pi\cap C_n = \emptyset$, since $f_0,f_1,f_2$ have no common zeroes.

We can consider the projection map $\pi: (\PP^n - \Pi) \rightarrow \Pi^\perp \cong \PP^2$ defined as:

\begin{equation}\label{pi}\pi(z_0:\ldots : z_n)\ =\ (z_0\ z_1 \ldots \ z_n) \cdot \begin{pmatrix}  a_{00}& a_{10}& a_{20} \cr a_{01}& a_{11}& a_{21} \cr &\vdots & \cr a_{0n} &a_{1n} & a_{2n} 
\end{pmatrix} \end{equation}
  
\noindent where $\Pi ^\perp \cong \PP^2$,  $\Pi^\perp =\langle F_0,F_1,F_2\rangle $, and $F_u=(a_{u0}:a_{u1}:\ldots :a_{un})$. 
 If we consider, in $ \Pi^\perp \cong \PP^2$, homogeneous coordinates $w_0,w_1,w_2$, with $(w_0:w_1:w_2) = w_0F_0+w_1F_1+w_2F_2 \in \PP^n$, we get that 
$$
\pi(z_0:\ldots :z_n)\ =\ (w_0:w_1:w_2)\quad {\rm with}\quad w_u = (z_0 \ldots z_n)\cdot F_u.
$$

Then it is immediate to check that the projection  $\pi(C_n)$ from $\Pi$ on the plane $\Pi ^\perp$  is exactly $C$, i.e. that we have ${\mathbf{f}}(s,t) = (\pi \circ \nu_n) (s,t)$, $\forall \, (s,t)\in \PP^1$.

Now, if we consider the equations (\ref{fk})  in $ \PP ^k\times \PP ^n$, we get a scheme $\tilde \Pi =p_2^{-1}(\Pi)$ and the intersection scheme   $Y' = Y \cap \tilde\Pi$, which is $(2k-4)$--dimensional (since $\dim Y =2k-1$ and $\mathbf{f}$ is a proper parameterization); we have that $p_1(Y') = \PP ^k$ for $k\geq 4$, while for $k=3$, $p_1(Y') = S_3 \subset \PP ^3$ is the  {\it Poncelet variety} associated to  $\langle f_0,f_1,f_2\rangle $ (e.g. see \cite{ISV}).

We are going to consider the $(n-k+4)\times (n+1)$ matrices:

\begin{equation}\label{trueMk}
M_k=  \begin{pmatrix} x_0&x_1&\ldots &x_k&0&0&\cdots &0 \cr
0&x_0&x_1&\cdots &x_k&0&\cdots&0 
\cr  & &\ddots&&  & & &  \cr
0&\cdots&0&x_0&x_1&\cdots &\cdots &x_k\cr
 a_{00}& a_{01}& a_{02}& a_{03}& \cdots& \cdots & a_{0n-1} &a_{0n}\cr
a_{10}& a_{11}& a_{12}& a_{13}& \cdots & \cdots & a_{1n-1} &a_{1n}\cr a_{20}& a_{21}& a_{22}& a_{23}& \cdots & \cdots& a_{2n-1} &a_{2n} 
\end{pmatrix}.
\end{equation}

For $k=3$, det$M_3$ defines a surface $S_3$  of degree $n-2$ in $\PP^3$; let us point out that in our paper \cite{BGI2} we used the singularities of such surface in order to investigate the presence of triple points on $C$; actually the use of the $0$-dimensional schemes $X_k$ we are going to define below is more efficient. 
\begin{defn}\label{defXk} Let $C\subset \mathbb{P}^2$ be a rational curve and $(f_0, f_1, f_2)$ be a proper parameterization.
For $2\leq k \leq n-1$, let  $X_k\subset \PP^k$ be the scheme defined by the $(n-k+3)$-minors of $M_k$ .
\end{defn}

Notice that for a generic rational curve $C\subset \PP^2$ of degree $n$,  the scheme $X_k$ will be empty for $k\geq 3$.

We want to use the schemes $X_k$, in order to study the singularities of $C$. The starting point for this project is the following result: 

\begin{prop}\label{Singpoint}
Let $C\subset \PP^2$ be a rational curve. The schemes $X_k$ introduced in Definition \ref{defXk} are either 0-dimensional or empty. Moreover: 
\begin{itemize}
\item $\forall \, k$, $2\leq k\leq n-1$, $X_k$ is non-empty iff there is at least a singular point on $C$ of multiplicity $\geq k$.

\item  Every singular point of $C$ yields at least a simple point of $X_2$ and $$\length X_2 =  {n-1\choose 2}$$ (notice that $X_2$ is never empty since $n\geq 3$). 
\end{itemize}
\end{prop}

\begin{proof}
There is a singular point of multiplicity at least $k$ on $C$ if and only if there is at least a  $(k-1)$-space $H\subset \PP^n$ that is $k$-secant to $C_n$ and whose intersection with the center of projection $\Pi$ has dimension $k-2$ ($H\cap C_n$ collapses to the singular point of $C$ under the projection from $\Pi$). Notice that $H$ need not to be such that $H\cap C_n$ is given by $k$ distinct points: in that case the singular point of $C$ has $k$ distinct branches. All we are asking is that the divisor $H\cap C_n$ has degree $k$. 
The dimension of $\Pi\cap H$ is $k-2$ if and only if the point $P_H\in \PP^k$ that parameterizes $H$ is such that the matrix $M_k$ has rank $n-k+2$ at $P_H$.  If $k=2$ this means that $P_H\in X_2$, while if $k\geq 3$ not only $P_H \in X_k$, but every $(j-1)$-subspace $H_i\subset H$ which is $j$-secant to $C_n$ will yield a point $P_{H_i}\in X_j$ (every different subscheme of length $j$ of $H\cap C_n$ will span such a $H_i$).

When $k\geq 3$, $H$ collapses to a point of $C$ if and only if $P_H\in X_k$.  Since $\Sing(C)$ is a finite set, for all $k\geq 2$, the scheme $X_k$ is 0-dimensional (or empty).

Since $n\geq 3$, $C$ cannot be smooth, so $X_2\subset \PP^2$ is 0-dimensional (and not empty) and its ideal $I_{X_2}$ has height 2 in $K[x_0,x_1,x_2]$ and it is defined by the $(n+1)$ (maximal) minors of a $(n+1)\times (n+2)$ matrix of forms. Hence, by \cite{refEN}, a minimal free resolution of $\mathcal{O}_{X_2}$ is given by the Egon-Northcott complex, as follows:
$$
0 \rightarrow   \mathcal{O}^{\oplus (n+1)}(-n+1)\rightarrow
 \mathcal{O}^{\oplus 3}(-n+1)\oplus \mathcal{O}^{\oplus (n-1)}(-n+2) \rightarrow \mathcal{O} \rightarrow \mathcal{O}_{X_2} \rightarrow 0,
$$

\noindent where $\mathcal{O}=\mathcal{O}_{\PP^2}$ and the second map is defined by $M_2$. From here we can conclude since the length of $X_2$ is $h^0(\mathcal{O}_{X_2})$, and, via twisting by $\mathcal{O}(n-1)$ and taking cohomology, we get: 

$$h^0(\mathcal{O}_{X_2}) = h^0(\mathcal{O}(n-1)) -3h^0(\mathcal{O})-(n-1)h^0(\mathcal{O}(1)) + (n+1)h^0(\mathcal{O}) = $$
$$ = {n+1\choose 2} -3 -3(n-1) +n+1 = {n-1 \choose 2},$$
as required. \end{proof}

\medskip

Recall that since $C$ has genus 0, the Clebsch formula gives $$ \deg X_2 = {n-1\choose 2}\ =\ \sum _{q}{m_q\choose 2},$$
where $q$ varies over all singularities on, and infinitely near, $C$. 

For all $P\in \Sing(C)$, we will indicate with  $$\delta_P = \sum_q {m_q\choose 2},$$  where $q$ runs over all $q$'s infinitely near $P$ (and  $\sum_q m_q \leq m_P$). The invariant $\delta_P$ measures the contribution of $P$ to the  Clebsch formula, and tells us that $P$ is equivalent to $\delta_P$ nodes for the genus count of $C$.

More algebrically, the number $\delta_P$ can also be defined like this:
 $$\delta_P = \length \left(\tilde{\mathcal{O}}_{C,P}/\mathcal{O}_{C,P}\right),$$
where $\mathcal{O}_{C,P}$ is the local ring of the structure sheaf $\mathcal{O}_{C}$ at the point $P$, and $\tilde{\mathcal{O}}_{C,P}$ is its integral closure. 

\section{Study of the set $\Sing(C)$.}

Our approach to the problem of finding and analyzing the singularities of a rational plane curve $C$ given parametrically uses the projection from $C_n\subset \PP^n$ which gives $C\subset \PP^2$ (by exploiting the equations of its center of projection $\Pi\subset \PP^n$ that are given by the parameterization of $C$) and the parameterization (by $\PP^k$) of the $k$-secant $(k-1)$-spaces of $C_n$.

\subsection{Cuspidal curves}

A first problem of particular interest that we will consider is how to determine when the curve $C$  is cuspidal, i.e. when all its singular points are cusps (i.e. uniramified singular points). This happens when the series of blowups which resolves the singularity  yields, at any singular point $P\in C$, only one point over $P$. Cuspidal rational curves on $\mathbb{C}$ are of particular interest since topologically they are spheres. Such curves have been widely studied (e.g. see \cite{MOE,FZ,O,Pi}).  The following proposition can be of interest since it gives a criterion to decide whether a given rational curve is cuspidal or not. We can also determine when we have only ordinary singularities, i.e. every singular point $P\in C$ with multiplicity $m$  comes via projection from $m$ distinct points $Q_1,\ldots,Q_m$ of $C_n$ and the tangent lines $T_{Q_i}(C_n)$ are such that $\pi(T_{Q_i}(C_n)) \neq \pi(T_{Q_{i'}}(C_n))$ for $i\neq i'$). This implies that there are no singularities infinitely near to $P$.

\medskip
\begin{prop}\label{cuspidal}
Let $C\subset \PP^2$ be a rational curve, given by a proper parameterization $(f_0:f_1:f_2)$, with $f_i\in K[s,t]_n$. Let  $\mathcal{C}_2$ be the conic $\{4x_0x_2-x_1^2=0\}$ and $X_2 \subset \PP^2$ be the 0-dimensional scheme defined in Definition \ref{defXk}. Then:

\begin{itemize}
\item $C$ is cuspidal if and only if $\Supp(X_2) \subset \mathcal{C}_2$.
 
\item $C$ has only ordinary singularities if and only if the scheme $X_2$ is reduced and $X_2\cap \mathcal{C}_2 = \emptyset$.
\end{itemize}
Moreover, in the first case, the number of singular points of $C$ is exactly the cardinality of $\Supp(X_2)$.
\end{prop}

\begin{proof}

The curve $C$ is cuspidal if and only if no secant line to $C_n$ gets contracted in the projection to $C$, but only tangent ones, i.e. if every point $Q\in \Pi\cap \sigma_2(C_n)$ lies on $\tau(C_n)$; this happens if and only if $R = p_1(p_2^{-1}(Q))$ belongs to $\mathcal{C}_2$, since the points of $ \mathcal{C}_2$ parameterize tangent lines inside $\sigma_2(C_n)$. Moreover, each cuspidal point of $C$  (regardless to its multiplicity) corresponds to a unique tangent to $C_n$ which is contracted by $\pi$, hence their number is given by $\Supp(X_2)\cap \mathcal{C}_2$.

On the other hand, if $C$ has only ordinary singularities, then no tangent line to $C_n$ gets contracted in the projection to $C$, (hence  $X_2\cap \mathcal{C}_2 = \emptyset$), so every singular point of multiplicity $m$ of $C$ comes from $m$ distinct points of $C_n$ (under the projection $\pi : C_n \rightarrow C$), and ${m \choose 2}$ secant lines of $C_n$ get contracted by $\pi$, yielding ${m \choose 2}$ points of $X_2$ (see Proposition \ref{Singpoint}).

Now, since $C$ has only ordinary singularities, by Clebsch formula we have ${n-1 \choose 2} = \sum_{P\in \Sing(C)}{m_P\choose 2}$, so, since $\length(X_2)= {n-1 \choose 2}$, and each $P\in  \Sing(C)$ yields ${m_P\choose 2}$ (distinct) points in $X_2$, the scheme $X_2$ has to be reduced.

If we know that $X_2$ is reduced and $X_2\cap \mathcal{C}_2=\emptyset$, instead, we have that there are no cuspidal points and we have that $\Supp(X_2)$ contains at most $\sum_{P\in \Sing(C)}{m_P\choose 2}$ points, so deg$X_2 ={n-1 \choose 2} \leq  \sum_{P\in \Sing(C)}{m_P\choose 2}$. But, by Clebsch formula,   ${n-1 \choose 2} \geq \sum_{P\in \Sing(C)}{m_P\choose 2}$, so $\deg X_2 ={n-1 \choose 2} =  \sum_{P\in \Sing(C)}{m_P\choose 2}$. Thus any singular point $P$ corresponds exactly to ${m_P\choose 2}$ secant lines to $C_n$ which get contracted by $\pi$, and they all sum up exactly to ${n-1 \choose 2}$;  this implies that there cannot be any infinitely near singularities, so all the singularities are ordinary ones.
\end{proof}

\begin{rem}
In order to use the previous proposition, one can use programs for symbolic computations, as CoCoA \cite{COCOA} or Macaulay \cite{m2}, like this:  given a parameterized curve $C\subset \PP^2$, use $f_0,f_1,f_2$ to write the matrix $M_2$ and compute the ideal $I_{X_2}$ defined by its maximal minors. Compute $J$, the radical ideal of $I_{X_2}$. Compute the ideal $J' :=J+(4x_0x_2-x_1^2)$.  If $J=J'$ then Supp$(C) \subset \mathcal{C}_2$, and $C$ is cuspidal.  If $J'$ is irrelevant and $I_{X_2}=J$, then $X_2\cap \mathcal{C}_2 =\emptyset$ and $X_2$ is reduced, so $C$ has only ordinary singularities (the fact that $J'$ is irrelevant, i.e. associated to the empty set, can be checked via its Hilbert function). 
\end{rem}

\medskip

\begin{ex}\label{6ticA2A4} Consider the following quartic curve $C\subset\mathbb{P}^2$:
$$\left\{ \begin{array}{l}
x=s^4 + s^3t \\
y= s^2t^2 \\
z= t^4
\end{array}
\right..
$$
Computing the ideal $I_{X_2}$ one finds that {$I_{X_2}=(x_0x_2-x_1x_2, x_1^2, x_0x_1)$}
and $X_2$ has length 3, while its radical is $J_{X_2}= (x_1,x_0x_2)$, hence 
$$\Supp(X_2) = \{(0:0:1), (1:0:0)\}\subset \mathcal{C}_2$$ and $C$ is cuspidal, with two cuspidal singular points (necessarily both of multiplicity 2). 
\end{ex}

\medskip

\begin{ex}\label{4ticA6} Consider the following quartic curve $C\subset\mathbb{P}^2$:
$$\left\{ \begin{array}{l}
x=s^4 + st^3 \\
y= s^2t^2 \\
z= t^4
\end{array}
\right..
$$
Computing the ideal $I_{X_2}{=(x_1^2, x_0x_1, x_0^2+x_1x_2)}$, one finds that $X_2$ has degree 3, while its radical is $J_{X_2}= (x_1,x_0)$, hence $\Supp(X_2) = \{(0:0:1)\}\subset \mathcal{C}_2$, hence $C$ is cuspidal, with one cuspidal singular point  and one can check that it is a point of multiplicity 2, since $X_3$ is empty. 
\end{ex} 

\medskip

\begin{ex}\label{6tic3triple} Consider the following sextic $C\subset\mathbb{P}^2$:
$$\left\{ \begin{array}{l}
x=4s^6 - 16s^5t + 3s^4t^2 + 28s^3t^3 - s^2t^4 - 6st^5 \\
y=4s^5t - 12s^4t^2 - 41s^3t^3 + 99s^2t^4 + 10st^5 - 24t^6 \\
z= s^5t - 3s^4t^2 - 13s^3t^3 + 27s^2t^4 + 36st^5
\end{array}
\right..
$$
Computing the ideal $I_{X_2}$, one finds that $I_{X_2}$ is a radical ideal, with support on 10 points, and that $I_{X_2}+(4x_0x_2-x_1^2)$ is irrelevant, hence none of the 10 points comes from cuspidal points. By Proposition \ref{cuspidal}, the curve $C$ has only ordinary singularities. 
\end{ex}

\subsection{Number of singularities}

Another problem of interest is to determine the number of singular points of $C$. What we will give here are algorithms that allow one to compute this number (and also how many singularities there are for each multiplicity). The algorithms depend also on the possibility to determine equations for the varieties which parameterize binary forms with given factorization.

The  main idea we will exploit is this: each singular point of $C$ of multiplicity $k$ is associated to one point in $X_k$ that parameterizes a $H\cong \mathbb{P}^{k-1}$, $k$-secant to $C_n$ which gets contracted by the projection $\pi$.  Notice that also each subspace $\widetilde{ H}\cong \mathbb{P}^{j-1} \subset H$ which is a $j$-secant to $C_n$ will be associated to a point in $X_j$, but it will not be associated to a singular point of $C$. So, if $j< k$ the points on  $X_j$ can be associated to singularities of multiplicity $j$, or come from those  $\widetilde{ H}$'s that are associated to singularities of $C$ of multiplicity $>j$. Therefore  in order to compute the number of singularities of $C$ of order $j$, we will have to distinguish the two kinds of points on $X_j$.

\begin{notation}\label{Not1} For any integer $k\geq 0$, we will consider the partitions    of $k$, and we will write a partition as an element $\overline{\lambda}=(\lambda_1,\ldots,\lambda_k) \in \mathbb{N}^k$, with 
$$\lambda_i\geq 0,\  \sum_{i=1}^k \lambda_i=k\quad \text{and} \quad  \lambda_i \geq \lambda_{i'} \quad  \text {if}\quad i\leq i'.$$

 We will consider the variety $\mathcal{R}_{\overline{\lambda} }\subset \PP^k$  where we view $\PP^k$ as the space parameterizing degree $k$ binary forms (i.e. $\PP^k = \PP (K[s,t]_k)$),
and  $\mathcal{R}_{\overline{\lambda} }$  parameterizes {the projective classes of } forms $G = L_1^{\lambda_1}\cdots L_k^{\lambda_k}$, where the $L_j$'s are linear forms:
{$$\mathcal{R}_{\overline{\lambda} }:=\left\{[G]\in\PP (K[s,t]_k)\, | \, G = L_1^{\lambda_1}\cdots L_k^{\lambda_k}, L_i \in K[s,t]_1 \right\}.$$}
 We know that $\mathcal{R}_{\overline{\lambda}}$ is an irreducible variety since it can be given parametrically.  

\medskip
\begin{rem}\label{idealRlambda}
It is not too hard to find the equations of $\mathcal{R}_{\overline{\lambda}}$: we can write the product $(a_1s+b_1t)^{\lambda_1}\cdots(a_ks+b_kt)^{\lambda_k}$ in the form $ \sum_{j= 0}^k \alpha_js^{k-j}t^j$, $\alpha_j\in K[a_1,b_1,\ldots,a_k,b_k]$ and then proceed with the elimination of the $a_i,b_i$'s from the ideal $(x_0-\alpha_0, \ldots, x_k-\alpha_k)$ in the polynomial ring $K[x_0,\ldots,x_k,a_1,b_1,\ldots,a_k,b_k]$. This yields the ideal $I_{ \mathcal{R}_{\overline{\lambda}}}$.
 \end{rem}

The set of partitions of $k$ is partially ordered, and we will consider $\overline{\lambda} \preceq \overline{\lambda'}$ if one can get $\overline{\lambda'}$ from $\overline{\lambda}$ by substituting some  $\lambda_j$, with a partition of it. \\ {For example, if we  use subscripts for repeated indexes,} we have $(5,3,2,1_8,0_7) \preceq (5,2,2,1_9,0_6)$, since we can get the second by substituting 3 with $(2,1)$ in the first one. 
\end{notation}

\begin{rem}\label{containementR}
{A variety  $\mathcal{R}_{\overline{\lambda}}$ is contained in $\mathcal{R}_{\overline{\lambda'}}$}, if and only if  $\overline{\lambda} \preceq \overline{\lambda'}$. 
\\
Clearly $\mathcal{R}_{(k,0,\ldots,0)} =  \mathcal{C}_k$  is contained in every $\mathcal{R}_{\overline{\lambda}}$; in particular, we have: 

\begin{equation}\label{containementRosc}
\mathcal{R}_{(k,0_{k-1})} =  \mathcal{C}_k \subset \mathcal{R}_{(k-1,1,0_{k-2})} \subset \mathcal{R}_{(k-2,1,1,0_{k-3})} \subset \cdots \subset  \mathcal{R}_{(2,1_{k-2},0)} \subset \mathcal{R}_{(1_{k})} =  \PP^k.
\end{equation}

Here  $\mathcal{R}_{(k-j,1_{j},0_{k-j-1})} = O^j(\mathcal{C}_k)$, the variety of $(j+1)$-osculating $j$-spaces to $\mathcal{C}_k$, i.e. $ O^j(\mathcal{C}_k)$ is the union:  $\cup_{P\in \mathcal{C}_k} O^j_P(\mathcal{C}_k)$ of the $j$-osculating spaces to $ \mathcal{C}_k $ (where $O^j_P(\mathcal{C}_k)$ is the linear span of the subscheme $(j+1)P\subset  \mathcal{C}_k $ in $\PP^k$). In particular,  $O^1(\mathcal{C}_k)= \tau(\mathcal{C}_k)$, the tangent developable of  $\mathcal{C}_k$.
\end{rem}

We will see that once we compute, for each $X_k$ of Definition \ref{defXk}, the cardinality of $\Supp(X_k)\cap \mathcal{R}_{\overline {\lambda}}$, for all partitions $\overline {\lambda}$ of $k$, we can also get an algorithm to compute the number of singular points of $C$ of given multiplicity.  

{\begin{notation}\label{Zk}
Let $Z_2,\ldots,Z_{n-1}$, $Z_i\subset \PP^i$, be the reduced schemes which are, respectively, the supports of $X_2,\ldots,X_{n-1}$ introduced in Definition \ref{defXk} and let $N_j$, for $j=2, \ldots , n-1$ be the number  of singularities of order $j$ of the curve $C\subset \mathbb{P}^2$.
\end{notation}}

\bigskip

{Notice that if $k$ is the maximum value such that $X_k\neq \emptyset$, then 
\begin{equation}\label{Nk}
N_k=\sharp Z_k
\end{equation}
(this will be the main content of Step \ref{Step1} in Algorithm \ref{trueAlg1}). Moreover, if $k'<k$,  then in $Z_{k'}$  we will find all the points associated  to all the singularities of order $j\geq k'$. 
To be more precise, in Proposition \ref{Algorithm1} we will show that if $\overline{\sigma}$ is a partition of $k-1$, then not all the points  of   $Z_{k-1}\cap \mathcal{R}_{\overline {\sigma}}$ are associated to  singularities of order $k-1$ of $C$. This happens when $\overline{\sigma}$ can be obtained from a partition $\overline{\lambda}$ of $k$,  such that $Z_{k}\cap \mathcal{R}_{\overline {\lambda}}\neq \emptyset$,  by subtracting  1 from some $\lambda_i$ (cf. Definition \ref{ancestor}).
}

\begin{ex} Consider the following quartic $C\subset \mathbb{P}^2$:
$$\left\{ \begin{array}{l}
x=s^4+t^4\\
y=s^4+s^2t^2+t^4 \\
z=s^3t
\end{array}\right. .
$$
Here we have that the ideal of the 5-minors of $M_2$ is $I_{X_2}=(x_2^2, x_0x_2, x_0x_1)$, hence $X_2$ has length 3 with support on 2 points, i.e. $X_2$ is the union of a simple point and of a scheme of length 2 supported at one point, so  $\length(Z_2)=2$.
The ideal of the 4-minors of $M_3$ is $I_{X_3}=(x_3, x_2, x_0)$ hence $X_3=Z_3$ is one simple point. Trivially, since $\deg C=4$, the scheme $X_4$ is empty. 
\\
Since $k=3$ is the maximum value such that $X_k\neq \emptyset$ and  $\length( Z_3)=1$ then $N_3=1$ and therefore $C$ has one triple point. Moreover we have that  $Z_3\cap  \mathcal{R}_{(2,1)}\neq \emptyset$, i.e. $Z_3$ has support on  $\mathcal{R}_{(2,1)}$.
\\
Now, it is trivial that a plane rational quartic with a triple point does not have any double point, but let us see this from the structure of $X_2$. If we count how many double points $C$ has, we have first to understand which are the $\mathcal{R}_{\overline{\lambda}}$ without trivial intersection with $Z_2$. Now $\length( Z_2 \cap \mathcal{R}_{(2,0)})=1$ and  $\length \left(Z_2 \cap \left(\mathcal{R}_{(1,1)}\setminus \mathcal{R}_{(2,0)} \right)\right)=1$. Clearly both $(1,1)$ and $(2,0)$ are partitions of 2 that can be obtained from $(2,1)$ by subtracting 1 to an entry: $(1,1)=(2-1,1)$ and  $(2,0)=(2,1-1)$. Therefore neither  $Z_2 \cap \mathcal{R}_{(2,0)}$ nor $ Z_2 \cap \mathcal{R}_{(1,1)}$ contribute to the singularities of order 2 of $C$. Hence we have re-discovered that $C$ does not have any double point.
\end{ex}

Before giving the details of the algorithm we need the following definition.

\begin{defn}\label{ancestor} We say that a partition $\overline{\lambda}=(\lambda_1, \ldots , \lambda_k)$ of $k$ is an \emph{ancestor} of a partition $\overline{\sigma}=(\sigma_1, \ldots , \sigma_{k-1})$ of $k-1$, if the $k$-uple $(\sigma_1, \ldots , \sigma_{k-1},0)$ can be obtained from the $k$-uple $(\lambda_1, \ldots , \lambda_k)$ by subtracting  1 to an entry $\lambda_i>0$ and reordering the entries in decreasing order. Moreover we define $n_{\overline{\lambda},\overline{\sigma}}$ to be the number of ways in which we can get $\overline{\lambda}$ as an ancestor of $\overline{\sigma}$. 
\end{defn}

\begin{ex} The partition $\overline{\lambda}=(3,3,1,1,0_4)$ of 8 is an ancestor of the partition $\overline{\sigma}=(3,2,1,1,0_3)$ of 7. Moreover $n_{\overline{\lambda},\overline{\sigma}}=2$, in fact $\overline{\lambda}$ can be seen as an ancestor  of $\overline{\sigma}$ in two different ways, by subtracting 1 to either of the first two entries and dropping the last zero. 
\end{ex}

We are now ready to describe the algorithm.

\begin{prop}\label{Algorithm1}
Let $C\subset \PP^2$ be a rational curve, given by a proper parameterization $(f_0:f_1:f_2)$, with $f_i\in K[s,t]_n$, and let $X_{n-1}, \ldots ,X_3,X_2$ be as in Definition \ref{defXk}. Then the following Algorithm \ref{trueAlg1}, based on the structure of the schemes $X_{n-1},\ldots,X_3,X_2$, computes the number $N$ of singular points of $C$. Moreover, Algorithm \ref{trueAlg1} computes also the number $N_k$ of singular points of multiplicity $k$. 
\end{prop}

\begin{proof} {Let $Z_k$ be the support of $X_k$ as in Notation \ref{Zk}.}
Before we give the steps for the algorithm, let us recall how the points in $Z_k\subset \PP^k$ are related to those in $Z_{k-1}\subset \PP^{k-1}$ (for $k\geq 3$). 

 If $R\in Z_k$, $k\geq 3$, $R=(a_0:\ldots :a_k)$ parameterizes a linear space $H_R\cong \PP^{k-1}$, $H_R\subset \PP^n$ which is $k$-secant to ${C}_n$ and which gets  contracted by $\pi$ to a singular point of $C$.
 
Actually, if $G_R=a_0s^k+a_1s^{k-1}t+\cdots+a_kt^k\in K[s,t]_k$ and $G_R = L_1^{\lambda_1}\cdots L_k^{\lambda_k}$ is a linear factorization of $G_R$, where $\overline{\lambda} = (\lambda_1,\ldots,\lambda_k)$ is a partition of $k$, this defines a divisor $D_R={\lambda_1}Q_1+\cdots+{\lambda_k}Q_k$ on ${C}_n$, whose linear span in $\PP^n$ is $H_R$. 

For all $i=1,\ldots ,k$ such that $\lambda_i>0$,
the divisor 
$D_{R,i}= \lambda_1Q_1+\cdots+(\lambda_i-1)Q_i+\cdots+\lambda_kQ_k$ 
defines a $(k-1)$-secant space $H_{i}\cong \PP^{k-2}$, which corresponds  to a point $R_i\in Z_{k-1}$ via the binary form $G_{i}=L_1^{\lambda_1}\cdots L_{i}^{\lambda_i-1}\cdots L_k^{\lambda_k}$. We can say that $D_{R,i}$ is defined by a partition
 $\overline{\sigma}=(\sigma_1, \ldots , \sigma_{k-1})$ of $k-1$ such that $\overline{\lambda}=(\lambda_1, \ldots , \lambda_k)$ is an ancestor of $\overline{\sigma}$ as in Definition \ref{ancestor}.

\begin{notation}\label{Z'} 
For each $k=3,\ldots ,n-1$, let $Z'_{k-1}$ be the subset of $Z_{k-1}$ made of those points (as $R_i$ above), which ``~come from $Z_k$~" (i.e. which correspond to a $H_{i}\subset H_R$, with $R\in Z_k$).

For each $R\in \Supp(Z_{k})$ defining a $k$-secant space $$H_R=\langle O^{\lambda_1-1}_{Q_1}(C_n),\cdots,O^{\lambda_k-1}_{Q_k}(C_n)\rangle,$$ 
 we have that,  for each $\lambda_i>0$, we can find  points $R_i\in Z_{k-1}$ such that 
 $$H_{R_i}=\langle O^{\lambda_1-1}_{Q_1}(C_n),\cdots,O^{(\lambda_i-1)-1}_{Q_i}(C_n),\cdots,O^{\lambda_k-1}_{Q_k}(C_n) \rangle.$$
\noindent Those are the points in $Z'_{k-1}$.
\end{notation}

All this description shows that if a point $R$ is in $Z_{k-1}'\subset Z_{k-1}$ then $R\in Z_{k-1}\cap\mathcal{R}_{\overline{\sigma}} $ for $\overline{\sigma}$ minimal with respect to the order defined above, and such that there exists an ancestor $\overline{\lambda}$  of $\overline{\sigma}$ (partitions of $k$ and of $k-1$ respectively) and  $Z_{k}\cap \mathcal{R}_{\overline{\lambda}}\neq \emptyset $. 

The points in $Z_{k-1}'$ must not be counted in determining  the amount of the singularities of order $k-1$. 

In the following algorithm we will count exactly all the singularities of a plane rational curve for each multiplicity, and Step \ref{Step3}  will compute the exact contribution of each $Z_j$ to the singularities of multiplicity $j$ by excluding  $Z'_{j}$, coming from points in some  $Z_k$, $k>j$, and some ancestor partition as just described.

\smallskip

With these notations settled, one can now read Algorithm \ref{trueAlg1} below.
 {\small{
\begin{algorithm}[H]\caption{Number of Singularities }\label{trueAlg1}
\INPUT: $M_i\in \mathbb{C}^{(n-k+4)\times (n+1)}$, $i=2, \ldots , n-1$ (see \eqref{trueMk}).
\OUTPUT Number of singular points of $C$ of multiplicity $k$.

\begin{enumerate}
\item \label{Step1}  Step \ref{Step1}.
\end{enumerate}
\begin{boxedminipage}{125mm}
    \begin{algorithmic}[1]
	\FOR {$ i =2, \ldots , n-1$}  
		\STATE \label{Step1b}   $I_{X_i}:=\left((n-i+3)\hbox{-minors of } M_i\right)$.	
	\ENDFOR
	\STATE\label{Step1a} {let $k$ be the biggest $i$ s.t. $I_{X_i}$ is not irrelevant.}
	\FOR {$ i =2, \ldots , k$}
		\STATE compute the radical ideals $I_{Z_i}:=\sqrt{I_{X_i}}$.\label{Step1c} 
	\ENDFOR
	\STATE\label{Step1d} Define  $N_k:=\sharp Z_k$(cf. \eqref{Nk}).
\end{algorithmic}
\end{boxedminipage}

\begin{enumerate}[resume*]
\item\label{Step2} Step \ref{Step2}.
\end{enumerate}

\begin{boxedminipage}{125mm}
\begin{algorithmic}[1]
	\STATE\label{Step2a} {let $\mathcal{R}_{{\overline{\lambda}}}\subset \mathbb{P}^k$ be the variety introduced in Notation \ref{Not1}}. 
	\FOR {any partition $\overline{\lambda}$  of $k$}
		\STATE\label{Step23} $N_{k,\overline{\lambda}} := \sharp \left((Z_k\cap \mathcal{R}_{\overline{\lambda}}) - \cup_{(\overline{\lambda}'\prec \overline{\lambda})} \mathcal{R}_{\overline{\lambda}'} \right)$, the number of points in $Z_k$ corresponding to binary forms of degree $k$ whose linear factorization has exactly the  $\lambda_i$'s as exponents (we have to exclude all the $\mathcal{R}_{\overline{\lambda}'}$ if $\overline{\lambda}'\prec \overline{\lambda}$, see Remark \ref{containementR}). 
	\ENDFOR
	\STATE Define  $$\mathcal{A}_k:=\{\overline{\lambda} \hbox{ partition of } k \,|\,N_{k,\overline{\lambda}}>0 \}.$$
\end{algorithmic}
\end{boxedminipage}

\begin{enumerate}[resume*]
\item\label{Step3} Step \ref{Step3}.
\end{enumerate}

\begin{boxedminipage}{125mm}
\begin{algorithmic}[1]

\FOR {any $\overline{\lambda}\in \mathcal{A}_k$}
\STATE set $\mathcal{B}_{\overline{\lambda}}:=\{\overline{\sigma} \hbox{ partition of } k-1 \,|\, \overline{\lambda} \hbox{ is an ancestor of }\overline{\sigma}\}$.\label{Step3a+}
\ENDFOR
\FOR {each partition $\overline{\sigma}\in \mathcal{B}_k$} 
\STATE\label{Step36} $N_ {k-1,\overline{\sigma}}$ analogously to {Step \ref{Step2}.\ref{Step23}.} \label{Step3a}
\ENDFOR
\FOR { all $\overline{\lambda}\in \mathcal{A}_k$ and all  $\overline{\sigma}\in \mathcal{B}_k$  such that $N_{k-1, \overline{\sigma}}>0$.}
\STATE
  $n_{\overline{\lambda}, \overline{\sigma}}$ as in Definition \ref{ancestor}  \label{Step3c}
\ENDFOR
\STATE\label{Step3d} let $N'_ {k-1,\overline{\sigma}} := N_ {k-1,\overline{\sigma}}- \sum_{\overline{\lambda}\in {A_k}} {N_{k, \overline{\lambda}}\cdot}n_{\overline{\lambda}, \overline{\sigma}}$ 
\STATE  define $$N_{k-1}: = \sum_{\overline{\sigma}\in \{\hbox{\footnotesize{partitions of} } k-1\}}  N'_ {k-1,\overline{\sigma}},$$ \label{Step312}
Then {$N_{k-1}=\sharp ( Z_{k-1} \setminus Z'_{k-1})$}.
\end{algorithmic}
\end{boxedminipage}

\begin{enumerate}[resume*]
\item\label{Step4} Step \ref{Step4}.
\end{enumerate}

\begin{boxedminipage}{125mm}
\begin{algorithmic}[1]
\STATE repeat Steps \ref{Step2} and \ref{Step3} in order to get $N_{k-2},\ldots,N_2$. 
\end{algorithmic}
\end{boxedminipage}

\begin{enumerate}[resume*]
\item\label{Step5} Step \ref{Step5}.
\end{enumerate}

\begin{boxedminipage}{125mm}
\begin{algorithmic}[1]
\STATE Compute
$$N= N_k\ +\ N_{k-1}+\cdots +\ N_2.$$
This is the number of singular points of $C$. Moreover, each $N_k$ is the number of singular points of $C$ of multiplicity $k$. 

\end{algorithmic}

\end{boxedminipage}
\end{algorithm}
}}
\end{proof}

\begin{rem}\label{neededforAlg1}

For the algorithm we need to find $(Z_k\cap \mathcal{R}_{\overline{\lambda}})$, hence we need to compute $I_{Z_k}+ I_{ \mathcal{R}_{\overline{\lambda}}}\subset K[x_0,\ldots,x_k]$. 

For the computation of the ideal $I_{ \mathcal{R}_{\overline{\lambda}}}$, see Remark \ref{idealRlambda}, while for $I_{Z_k}$ we have the determinantal ideal of $X_k$, so we just compute its radical, e.g. via COCOA  \cite{COCOA} or Macaulay 2 \cite{m2}. With the same programs we can compute the sum of the two ideals.
 
For the reader who is familiar with numerical computations, notice that if one knows the coordinate of $Z_k$, a way of testing if an element belongs to $\mathcal{R}_{\overline{\lambda}}$  is given, for a particular example, in  \cite[\S 5.6]{BDHM} via homotopy continuation with \cite{Bertini}.
\end{rem}

\begin{ex}[This is Algorithm \ref{trueAlg1} on Example \ref{6tic3triple}]\label{exBis} Let $C\subset \mathbb{P}^2$  be  the sextic curve already considered in Example \ref{6tic3triple}.

    \begin{itemize}
\item [Step \ref{Step1}:]
We do not write $I_{X_3}$ here for brevity,
anyway one can see that the scheme $X_3$ has degree 3 and it is reduced.
\item [Step \ref{Step1}.\ref{Step1a}:]
The biggest $k$ such that $X_k$ is not empty is $k=3$.
\item [Step \ref{Step1}.\ref{Step1c}:] 
$I_{Z_3}=\sqrt{I_{X_3}}=(x_0-15623x_1-252x_2+1386x_3,x_1x_3+9137x_2x_3-11429x_3^2,x_1x_2+10667x_2^2+5210x_2x_3+7937x_3^2,x_1^2-14224x_2^2-13900x_2x_3-
      3410x_3^2)$ .
\item[Step \ref{Step1}.\ref{Step1d}] $K[x_0.x_1,x_2,x_3]/I_{Z_3}$ has Hilbert polynomial equal to 3, hence $N_3=3$ and $C$ possesses three triple points.

\item [Step \ref{Step2}:] 
The intersection of $X_3$ with both $\mathcal{C}_3=\mathcal{R}_{(3,0,0)}$ and $\tau(\mathcal{C}_3)= \mathcal{R}_{(2,1,0)}$ is empty (Step \ref{Step2}.\ref{Step2a}),  
 hence we have to keep track only of $\overline{\lambda}=(1,1,1)$ since $N_{3,(1,1,1)}=3$.
\item [Step \ref{Step3}.\ref{Step3a+}:]
The only partition $\overline{\sigma}$ of 2 which has as
 $\overline{\lambda}=(1,1,1)$ as ancestor, is $\overline{\sigma}=(1,1)$, i.e. $\mathcal{B}_{(1,1,1)}=\{(1,1)\}$.
\item [Step \ref{Step3}.\ref{Step3a}:] 
Since $ \mathcal{R}_{(1,1)}= \mathbb{P}^2$, in order to compute $N_{2,(1,1)}$ we have to understand how $Z_2$ intersects $\mathcal{R}_{(2,0)}$. One can easily compute that $\length (X_2)=\length (Z_2)=10$ and that $X_2\cap \mathcal{R}_{(2,0)}=\emptyset$. Hence  $N_{2,(1,1)}=10-0=10$.
\item [Step \ref{Step3}.\ref{Step3c}:] 
We have that  $n_{(1,1,1),(1,1)}=3$.
\item [Step \ref{Step3}.\ref{Step3d}:] 
$N'_{2,(1,1)}=N_{2,(1,1)}-N_{3,(1,1,1)}\cdot n_{(1,1,1),(1,1)}=10-3\cdot 3=1$,
\item[Step \ref{Step3}.\ref{Step312}:]
$N_{2}=1$. Hence 9 of the 10 simple points of $X_2$ come from $X_3$, and $C$ has only one double point. 
\item[Step \ref{Step4}:] There is no need to run this step in this example.
\item [Step \ref{Step5}:]  $N=3+1=4$: the curve $C$ has 4 singular points:  three ordinary triple points and an ordinary node (the fact that the they are ordinary singularities is a consequence of Proposition \ref{cuspidal}, see Example \ref{6tic3triple}).
\end{itemize}
\end{ex}

Actually, Algorithm \ref{trueAlg1} says more about the singularities of $C$ than just their number and multiplicities.

\begin{defn}\label{branchDef} Let $P\in \Sing C$; let $Q_1, \ldots , Q_s\in C_n$ be the points such that $\pi|_{C_n}(Q_i)=P$, then we say that $C$ has $s$ branches at $P$. Moreover let $\left(\pi|_{C_n}\right)^{-1}(P)=\lambda_1Q_1 + \cdots + \lambda_sQ_s$ as divisors in $C_n$, then we say that the $i$-th branch of $C$ has multiplicity $\lambda_i$ at $P$.
\end{defn}

\begin{cor}\label{branches}
Let $Z_j$ and $Z'_j$ as in Notation \ref{Zk} and \ref{Z'} respectively. If $R\in Z_j \setminus Z_j'$ and  $\overline{\lambda} =(\lambda_1,\ldots ,\lambda_j)$ is minimal (with respect to the order defined above) in order to have $R\in \mathcal{R}_{\overline{\lambda}}$, then the number of branches of $C$ at the singular point $P\in C$ associated to $R$ (i.e. $P=\pi(H_R)$) is the number  $l\left(\overline{\lambda}\right)$, the number of the $\lambda_i$'s different from zero in the partition  $\lambda$.
Moreover, for $i=1,\ldots ,l(\overline{\lambda})$, the $i$-th branch has multiplicity $\lambda_i$ at $P$.
\end{cor}

\begin{proof}
If $H_R = \langle O^{\lambda_1-1}_{Q_1}(C_n),\ldots ,O^{\lambda_j-1}_{Q_j}(C_n)\rangle $, then for each $Q_i$ we have $\pi(O^{\lambda_i-1}_{Q_i}(C_n))=P$. Hence each $Q_i$ gives a different branch of $C$ at $P$, whose multiplicity is actually $\lambda_i$.
\end{proof}

\begin{algorithm}[H]\renewcommand{\thealgorithm}{1.1}\caption{Branch Structure}\label{AggiuntaAlg1}
\textbf{Input:} $M_i\in \mathbb{C}^{(n-k+4)\times (n+1)}$, $i=2, \ldots , n-1$ (see \eqref{trueMk}).
\\
\textbf{Output:} Number of branches.
\begin{boxedminipage}{125mm}
    \begin{algorithmic}[1]
    \STATE Repeat Algorithm \ref{trueAlg1}, from Step \ref{Step1} to Step \ref{Step3}.\ref{Step3d} in order to get all $N_{k,\overline{\lambda}}$ (of Algorithm \ref{trueAlg1} Step \ref{Step2}.\ref{Step23}) and  $N'_{k-1,\overline{\sigma}}$ (of Algorithm \ref{trueAlg1} Step \ref{Step3}.\ref{Step36}) for every partition $\overline{\lambda}$ of $k$ and $\overline{\sigma}$ of $k-1$;
    \FOR {$j=(k-2),\ldots,2$}
\STATE Repeat Algorithm \ref{trueAlg1} from Steps \ref{Step2} to \ref{Step3}.\ref{Step3d} in order to get all $N'_{j,\overline{\sigma}}$ for every $j=k-2,\ldots , 2$ and any partition $\overline{\sigma}=(\sigma_1, \ldots , \sigma_j)$ of $j$; 
\ENDFOR
\STATE Set $N'_{k,\overline{\lambda}}:=N_{k,\overline{\lambda}}$;
\STATE Each $N'_{j,\overline{\sigma}}$ gives the number of points in $\Sing (C)$ having multiplicity $j$ and $l(\overline{\sigma})$ branches with multiplicity $\sigma_i$, $i=1, \ldots , l(\overline{\sigma})$, at each branch.
    \end{algorithmic}
\end{boxedminipage}

\end{algorithm}

\medskip

\begin{ex}\label{5tic4plebibranched}
Let $C$ be the quintic curve defined by:
$$\left\{ \begin{array}{l}
x=s^5 + s^3t^2 - s^2t^3+t^5 \\
y=s^3t^2 + s^2t^3 \\
z= s^3t^2 - s^2t^3 
\end{array}
\right..
$$
Computing the ideal $I_{X_4}$, one finds that $I_{X_4}$ is a radical ideal, with support at the point $R=(0:0:1:0:0)$. Since $\deg(C)=5$, the point $P\in \Sing(C)$, of multiplicity 4, corresponding to $R$ is the only singular point of $C$. We have that $R$ parameterizes the form $s^2t^2$, and $H_R\cong \PP^3$ (see Notation \ref{Z'}) is a 4-secant space spanned by the divisor (on $C_5$), $2(1:0:0:0:0:0)+2(0:0:0:0:0:1)$. Moreover, $R\in \mathcal{R}_{\overline{\lambda}}$, with $\overline{\lambda}= (2,2,0,0)$. Since $N_{4,\overline{\lambda}}=1$, and $l\left(\overline{\lambda}\right)=2$, we have that
 $P\in C$ has two branches, each of them of multiplicity 2 at $P$ (at $P$, the curve $C$ appears as the union of two different cusps).  
\end{ex}

\medskip

\begin{notation} If $C$ is a plane rational curve of degree $n$, we set, for each $k=2,\ldots,n-1$: 
$$\Sing_k(C) := \{P\in C\ | \ P\ \hbox{ is\ singular\ of\ multiplicity\ }k\} \subset \PP^2.$$ 
\end{notation}

There is another way to compute the number of singularities of $C$, another algorithm which gives directly the ideals, in $\PP^2$, of the (reduced) sets $\Sing(C)$ and Sing$_k(C)$, $k=2,\ldots , n-1$. 

\bigskip

\begin{prop}\label{Algorithm2} Let $C$, $X_{n-1},\ldots,X_3,X_2$ be as in Definition \ref{defXk}. Then Algorithm \ref{Algorithm2True}, based on the structure of $X_{n-1},\ldots,X_3, X_2$ and on the projection $\pi :\ C_n \rightarrow C$, computes the radical ideals of the reduced sets $\Sing(C)$, $\Sing_2(C),\ldots , \Sing_{n-1}(C)$ and also allows one to compute their cardinalities $N, N_2,\ldots,N_{n-1}$. 
\end{prop}

\begin{proof}
Let us sketch the steps of the algorithm.
\end{proof}

{\small{
\begin{algorithm}[H]\setcounter{algorithm}{1}\caption{Ideals of Singularities}\label{Algorithm2True}
\textbf{Input}: $M_i\in \mathbb{C}^{(n-k+4)\times (n+1)}$, $i=2, \ldots , n-1$ (see \eqref{trueMk}).\\
\textbf{Output 1}: Ideal of $\Sing(C)$, hence also number $N$ of singular points of $C$.
\\
\textbf{Output 2}: Ideals of $\Sing_k(C)$, hence also numbers $N_k$ of singular points of $C$ with multiplicity $k=2,\ldots n-1$.
\begin{boxedminipage}{125mm}
    \begin{algorithmic}[1]
\STATE\label{Alg2Step1}  Compute the ideal $I_{X_2}\subset K[x_0,x_1,x_2]$ given by the maximal minors of $M_2$;
\STATE\label{Alg2Step12} Compute the ideal $I_{Z_2}:=\sqrt{I_{X2}}$ which defines the scheme $Z_2\subset \PP^2$, the (reduced) support of $X_2$;
\STATE\label{Alg2Step13} Consider $I_{Z_2}\subset K[x_0,x_1,x_2,z_0,\ldots,z_n]$;
\STATE\label{Alg2Step2}   Compute the ideal $U\subset K[x_0,x_1,x_2,z_0,\ldots,z_n]$ generated by the equations \eqref{eqn1}: $\sum_{i=0}^k x_iz_{i+j} = 0, \; \; j=0, \ldots , n-k;$
\STATE\label{Alg2Step25} Compute the ideal $I:=I_{Z_2}+U\subset K[x_0,x_1,x_2,z_0,\ldots,z_n]$, 
(this gives the fibers $p_1^{-1}(Z_2)$ in $Y$);
\STATE\label{Alg2Step22} Compute the saturation ideal $I'\subset K[x_0,x_1,x_2,z_0,\ldots,z_n]$ of $I$ with respect to $(x_0,x_1,x_2)$;
\STATE\label{Alg2Step3}   Compute the elimination ideal $I_L\subset K[z_0,\ldots,z_n]$ obtained by eliminating the variables  $x_0,x_1,x_2$ from  the ideal $I'$  (notice that $I_L$ is the ideal of the lines in $\PP^n$ parameterized by $Z_2$. These lines are exactly the secant and tangent lines to $C_n$ which get contracted via $\pi$ to the singular points of $C$);
\STATE\label{Algo2Step8} Consider $I_L\subset K[z_0,\ldots,z_n,w_0,w_1,w_2]$;
 \FOR {$u=0,1,2$}
 \STATE $G_u :=
w_u - a_{u0}z_0-\cdots-a_{un}z_n$ (see Equation \eqref{fk} and after, in the previous section)
\ENDFOR
\STATE Compute the ideal $G:=(G_0,G_1,G_2)\subset K[z_0,\ldots,z_n,w_0,w_1,w_2]$;
\STATE\label{Alg2Step4}  Compute the ideal $I'':=I_L+G\subset K[z_0,\ldots,z_n,w_0,w_1,w_2]$;
\STATE\label{Alg2Step5}  Compute the elimination ideal $I_N\subset K[w_0,w_1,w_2]$ obtained by eliminating the variables $z_0,\ldots,z_n$ from $I''$
(notice that $I_N$ is the ideal of the image via the projection $\pi$ of the lines obtained in Step \ref{Alg2Step3});
\STATE\label{Alg2Step52} Compute the ideal  $J_N:=\sqrt{I_N}$, i.e. the ideal of the (reduced) set $\Sing(C)$;
\STATE\label{Alg2Step6}  Compute the Hilbert polynomial  of $ K[w_0,w_1,w_2]/J_N$ in order to get the number $N$ of the singularities of $C$;
\STATE{If what was needed was the total number of singularities, the algorithm can STOP HERE. Otherwise one continues with next step}
\FOR {$k=3,\ldots,n-1$}
\STATE \label{Alg2Step7}   the same process of Steps \ref{Alg2Step1}--\ref{Alg2Step5}, starting with the ideal $I_{X_k}\subset K[x_0,\ldots,x_k]$, thus obtaining radical ideals $J_{\geq k}\subset  K[w_0,w_1,w_2]$, of the reduced set  of the singularities of $C$ of multiplicity at least $k$;
\ENDFOR
\FOR {$k=3,\ldots,n-1$}
\STATE\label{Alg2Step8}    Compute the ideal $J_k:=J_{\geq k}:J_{{\geq k+1}}$ to get the ideal of $\Sing_k(C)$, the set of singularities of order $k$;
\ENDFOR
\FOR {$k= 2,\ldots, n-1$}
\STATE\label{Alg2Step9} 
compute the Hilbert polynomial   of $K[w_0,w_1,w_2]/J_k$ in order to get the number $N_k$ of the singularities of $C$ of order $k$.  
\ENDFOR
\end{algorithmic}
\end{boxedminipage}
\end{algorithm}
}}

In order to help the reader in understanding similarities and differences between the above two algorithms we run Algorithm \ref{Algorithm2True} again on Example \ref{6tic3triple} that is the same as Example \ref{exBis}.

\begin{ex}\label{exTris}
Let $C\subset \mathbb{P}^2$  be  the sextic curve already considered in Examples \ref{6tic3triple} and  \ref{exBis}; let us analyze it using  Algorithm \ref{Algorithm2True}. We will give the steps of the computation in a Macaulay 2 \cite{m2} code to illustrate how one can effectively do the computations. At every step we specify
 what the command is doing. 

\begin{enumerate}[label=(\alph*)]
\item\label{primo}
	$R=QQ[x_0..x_2];\quad $ \ - -  This is the ring to use in $\PP^2$.
\item $f_0=4*s^6 - 16*s^5*t + 3*s^4*t^2 + 28*s^3*t^3 - s^2*t^4 - 6*s*t^5$; 
$f_1=4*s^5*t - 12*s^4*t^2 - 41*s^3*t^3 + 99*s^2*t^4 + 10*s*t^5 - 24*t^6$; 
	$f_2= s^5*t - 3*s^4*t^2 - 13*s^3*t^3 + 27*s^2*t^4 + 36*s*t^5$; \quad - - These are the equations giving the parameterization of $C$.
\item $M2=\mathrm{matrix}\{\{x_0,x_1,x_2,0,0,0,0\},\{0,x_0,x_1,x_2,0,0,0\},
 	\{0,0,x_0,x_1,x_2,0,0\},\{0,0,0,x_0,x_1,x_2,0\},$\\$\{0,0,0,0,x_0,x_1,x_2\}
    	,f_0,f_1,f_2\};$\quad - - This yields the matrix $M_2$ defining $X_2$.

	\item  $IX2=\mathrm{minors}(7,M2);$ \quad - - The ideal $I_{X_2}$.
	\item $IZ2=\mathrm{radical} ( IX2);$ \quad - - The ideal $I_{Z_2}$.
	\item $\mathrm{hilbertPolynomial}(IX2)$ 
	\\
	- - Answer: $10P_0$. \quad - - This implies $\length X_2=10$.
	\item $\mathrm{hilbertPolynomial}(IZ2)$
	\\
	- - Answer: $10P_0$. \quad  - - This implies that $\length X_2 =\length (Z_2)=10$ i.e. $X_2$ is reduced and contains 10 simple points.
	\item $S=QQ[x_0..x_2,z_0..z_6];$ \quad - - This is the ring of  $\PP^2\times \PP^6$.
	\item $I=\mathrm{sub}(IZ2,S);$ \quad - - $I$ is $I_{Z_2}$ as an ideal in the above ring (Step \ref{Alg2Step13}).
	\item $J=I+\mathrm{ideal}(x_0*z_0+x_1*z_1+x_2*z_2,x_0*z_1+x_1*z_2+x_2*z_3,x_0*z_2+x_1*z_3+x_2*z_4,x_0*z_3+x_1*z_4+x_2*z_5,x_0*z_4+x_1*z_5+x_2*z_6);$\quad  - - $J$ is the ideal of $Y\cap p_1^{-1}(Z_2)$ (Step 5), using equations (\ref{eqn1}).
	\item $I'=\mathrm{saturate}(J,\mathrm{ideal}(x_0,x_1,x_2));$ \quad - - (Step \ref{Alg2Step22}).
	\item $H=\mathrm{eliminate}(\mathrm{eliminate}(\mathrm{eliminate }(I',x_0),x_1),x_2);$\quad  - - We eliminate the $x_i$'s variables, in order to get an ideal $H$ in the $z_j$'s. 
	\item $T=QQ[z_0..z_6];$\quad - - The ring of coordinates of $\PP^6$.
	\item $IL=\mathrm{sub}(H,T);$\quad  - - We read the ideal $H$ as an ideal $I_L$ giving the secant lines to $C_6\subset \PP^6$ parameterized by $Z_2$ (Step \ref{Alg2Step3}).
	\item\label{o} $\mathrm{hilbertPolynomial} \; \mathrm{Proj}(T/IL)$
	\\- - Answer: $-9P_0+10P_1$.\quad  - - This implies that there are exactly 10 lines in $\mathbb{P}^6$ that get contracted to singular points of $C$.
 	\item $ZW=QQ[z_0..z_6,w_0..w_2];$\quad - - The ring we will use to project from $\PP^6$ to $\PP^2$.
	\item $G0=-4*z_0+16*z_1-3*z_2-28*z_3+z_4+6*z_5+w_0$;\\
$G1=-4*z_1+12*z_2+41*z_3-99*z_4-10*z_5+24*z_6+w_1$;\\
$G2=-z_1+3*z_2+13*z_3-27*z_4-36*z_5+w_2$;

 - - These equations allow one to see the $\PP^2$ containing $C$ inside $\PP^6$ as $\Pi ^\perp$.
\item $IZW=\mathrm{sub}(IL,ZW);$\quad - - Here we read the ideal $I_L$ of the ten lines in the new ring.
\item $I''=IZW+\mathrm{ideal}(G0,G1,G2);$  \quad - - We intersect the lines with the $\PP^2$.
	\item $IE=\mathrm{eliminate}(\mathrm{eliminate}(\mathrm{eliminate}(\mathrm{eliminate}(\mathrm{eliminate}(\mathrm{eliminate}(\mathrm{eliminate}(I'',z_0),z_1),z_2),z_3),z_4),$ \\$z_5),z_6);$ \quad - - We perform elimination to find the points in $\PP^2$ (Step \ref{Alg2Step5}).
	\item $W=QQ[w_0..w_2];$ \quad - - The ring of coordinates of $\PP^2$.
	\item $IN=\mathrm{sub}(IE,W);$ \quad - - The ideal of the singular points of $C$ in $\PP^2$.
	\item $JN=\mathrm{radical} (IN)$;\quad - - Its radical.
	
\item $\mathrm{hilbertPolynomial} \; JN$ \\
	- - Answer: $4P_0$. \quad - - This implies that $N=4$, i.e. $C$ has 4 singular points (Step \ref{Alg2Step6}).
	\end{enumerate}
	We stop here with this example, at Output 1, because we already know that $N_3=3$ and $N_2=1$.
\end{ex}

\medskip 
 \begin{rem} Let $C,N_2,\ldots,N_{n-1}$ be as in the Propositions above. Then $C$ does not possess  singular points infinitely near to other singular points if and only if
 $${n-1 \choose 2 } = \sum_{k=1}^{n-1} N_k{k\choose 2}.
 $$
 \end{rem}
 
  \medskip

This is an immediate consequence of the Clebsch formula, hence Algorithms \ref{trueAlg1} or \ref{Algorithm2True} may be used to check if there are infinitely near singularities.

\subsection{Coordinates of singularities}
 
Let us notice that  Algorithm \ref{Algorithm2True}  gives the ideal of the singular points of $C$, but if one is interested in computationally finding the coordinates of those points, their ideal may not be the best way to get them: it could be difficult to use numerical approximations. Things could be easier if we get the points in $\PP^1$ such that their images via the parameterization $(f_0:f_1:f_2)$ give the singular points. The following algorithm does that.

\begin{prop}\label{Algorithm3} Let $C, {\mathbf{f}}$ be as in the previous section. Then the  following algorithm computes the points in $\PP^1$ whose images via $\mathbf{f}$ are the singular points in $\Sing_j(C)$, for each $j$. 
\end{prop}
\begin{algorithm}[H]\caption{Coordinates of Singularities of $C$}\label{Algorithm3True}
 \textbf{Input}: $M_i\in \mathbb{C}^{(n-k+4)\times (n+1)}$, $i=2, \ldots , n-1$ (see \eqref{trueMk}).\\
\textbf{Output}: Points of each $\Sing_j(C)$ and their coordinates.

\begin{boxedminipage}{125mm}

    \begin{algorithmic}[1]

	\STATE\label{Alg3Step1}  Same as Step \ref{Alg3Step1} of Algorithm \ref{trueAlg1}: find $k= \max \{j \ |\ X_j \neq \emptyset\}$;
	\STATE\label{Alg3Step2}  Same as Steps from  \ref{Alg2Step1} to \ref{Alg2Step3} of Algorithm \ref{Algorithm2True} but starting from $I_{X_k}\subset K[x_0,\ldots,x_k]$, so to find the ideal $I_{L_k}\subset K[z_0,\ldots z_n]$ of the scheme $L_k\subset \PP^n$ which is the union of the $k$-secant  $(k-1)$-spaces of $C_n$ which get contracted by $\pi$;
	\STATE\label{Alg2Step2} Compute the ideal $I_{C_n} \subset K[z_0, \ldots , z_n]$ of the rational normal curve (given by the $2\times 2$-minors of the matrix $B_1$ in Equation (\ref{Bk}));
 	\STATE\label{Alg3Step4} Compute the ideal $I_{S_k} := I_{L_k}+I_{C_n}\subset K[z_0, \ldots , z_n]$ of the scheme $C_n\cap L_k$;
 	\STATE Consider $I_{S_k}\subset K[z_0,\ldots,z_n,s,t]$;
 	\STATE Compute the ideal $I_k:=I_{S_k}+(z_0-s^n,z_1-s^{n-1}t,\ldots,z_n-t^n)\subset K[z_0,\ldots,z_n,s,t]$;
	\STATE Compute the saturation ideal $J_k\subset K[z_0,\ldots,z_n,s,t]$ of $I_k$
 with respect to $z_0,\ldots,z_n$;
 	\STATE\label{Alg3StepF} Compute the (principal) elimination ideal $(F_k)\subset K[s,t]$  obtained  by eliminating the variables $z_0,\ldots,z_n$ from $J_k$;
	\STATE\label{Alg3Step7} Compute the zeroes of $F_k$ (this can be done by approximation), they give the points in $\PP^1$ whose image via ${\mathbf{f}}$ give $\Sing_k(C)$ (the points corresponding to non-ordinary $k$-secant spaces in $L_k$ will appear with higher multiplicity  as solutions of $F_k$);
 	\STATE\label{Alg3Step8} Apply ${\mathbf{f}}$ to the points found in Step \ref{Alg3Step7}  and get the (approximated) coordinates of the points of $\Sing_k(C)$ in $\PP^2$;
 	 \FOR {$j=k-1,\ldots ,2$,} \STATE \label{Alg3Step9} repeat Steps \ref{Alg3Step2} to \ref{Alg3StepF}, in order to find the form $F_j\in K[s,t]$ whose zeroes are the points in $\PP^1$ associated to $C_n\cap L_j$; since $F_{j+1}$ divides $F_j$,  determine (maybe by approximation) the zeroes of $F_j/F_{j+1}$; apply   ${\mathbf{f}}$ to such points, thus determining the coordinates of the points in $\Sing_j(C)$ for each $j$.
	\ENDFOR
    \end{algorithmic}
\end{boxedminipage}

\end{algorithm}

\begin{rem} Notice that if one only wants to determine the coordinates of the points in $\Sing(C)$, regardless to their multiplicities, then one can perform Algorithm \ref{Algorithm3True} just starting with $X_2$, instead of $X_k$; in this way one determines the form $F_2\in K[s,t]$ whose zeros will give (via the parameterization) all the points in $\Sing(C)$.
\end{rem}

\begin{ex}\label{exQuattro} This is Algorithm \ref{Algorithm3True} on Examples \ref{6tic3triple}, \ref{exBis} and \ref{exTris}.

Let $C\subset \mathbb{P}^2$  be  the sextic curve already considered in Examples \ref{6tic3triple}.
 
\begin{enumerate}\setcounter{enumi}{11} \setlength\itemsep{1em}
\item[Step 1] We already know that $k=3$ here.
\item[Step 2]Perform step \ref{Alg3Step2} to find the ideal $I_{L_3}$ of the 3-secant planes to $C_6$ (we do not write down the ideal here).
 	\item[Step \ref{Alg2Step2}]:  $B_1=\left( \begin{array}{cc} z_0& z_1 \\z_1& z_2\\ z_2& z_3\\z_3& z_4\\z_4& z_5\\z_5&    z_6
\end{array}     \right)$; so

   $I_{C_6}=(-z_1^2+z_0z_2,-z_1z_2+z_0z_3,-z_2^2+z_1z_3,-z_1z_3+z_0z_4,-z_2
     z_3+z_1z_4,-z_3^2+z_2z_4,-z_1z_4+z_0z_5,-z_2z_4+z_1z_5, -z_3z_4+z_2
     z_5,-z_4^2+z_3z_5,-z_1z_5+z_0z_6,-z_2z_5+z_1z_6,-z_3z_5+z_2z_6,-z_
     4z_5+z_3z_6,-z_5^2+z_4z_6)
$;
\item[Step \ref{Alg3Step4}:] The Hilbert polynomial of $I_{S_3}$ is $H_{I_s}(t)=9$ then we have that  
 $\length \left((\pi|_{C_n})^{-1}(\Sing_3(C))\right)=9$.
\item[Step \ref{Alg3StepF}:] $(F_3)=(-s^8t + 5s^7t^2 + 29/4s^6t^3 - 217/4s^5t^4 + 65/4s^4t^5 + 341/4s^3t^6 - 9/2s^2t^7 - 18st^8)$;
      \item[Step \ref{Alg3Step7}] We computed the zeros  of $F_3$ with Bertini \cite{Bertini} (maximum precision utilized $10^{-52}$, less decimal digits shown here):
    $$  \footnotesize{\begin{array}{l|l}
      s&t\\ \hline
      0&1\\
      1&-1\\
1& 3.333e^{-1} +i2.775e^{-17} \\
\hline 
1&0\\
1&0.25\\
1&-3.333e^{-1} -i8.326e^{-17}\\ \hline
1&0.5\\
1&2\\
1&-2\\ 
      \end{array}}$$

\item[Step \ref{Alg3Step8}:] {{\begin{enumerate}
\item\label{numa}$\mathbb(f)(0:1)= (0:-24:0)$;

\item\label{numb}$\mathbb(f)(1: -1)
=(0:90:0)$;

\item\label{numc}$\mathbb(f)(1: 3.333e^{-1}+i 2.775e^{-17})
=(1.172e^{-15}-i1.439e^{-16}:
-0.288-i8.258e^{-17}:\\
-2.220e^{-16}+i2.467e^{-17})$;

\item\label{numd}$\mathbb(f)(1:0)
= (4:0:0)$;

\item\label{nume}$\mathbb(f)(1: 0.25)
= (0.6:0:0)$;

\item\label{numf}$\mathbb(f)(1:-3.333e^{-1} -i8.326e^{-17})
=\\(8.641+i7.401e^{-16}:\\
-1.040e^{-16}+i1.259e^{-15}:
 i2.590e^{-16})$;

\item\label{numg}$\mathbb(f)(1: 0.5)
= (0:0:0.9375)$;

\item\label{numh}$\mathbb(f)(1: 2)
= (0:0: 1470)$;

\item\label{numi}$\mathbb(f)(1:-2)
=(0:0:-630)$;
\end{enumerate}
}}
\item[Step 10:] Notice that, with a precision of $10^{-52}$, 
\begin{itemize}
\item the cases \eqref{numa}, \eqref{numb}, \eqref{numc} lead to the same projective point: $(0:1:0)$;
\item the cases \eqref{numd}, \eqref{nume}, \eqref{numf} lead to the same projective point: $(1:0:0)$;
\item the cases \eqref{numg}, \eqref{numh}, \eqref{numi} lead to the same projective  point: $(0:0:1)$.
\end{itemize}

These are the coordinates of the 3 triple points.

\item[Step 11:]
Perform the steps above starting with $X_2$, so to get 

\smallskip

 $F_2=-s^{10}t + 5092/713s^9t^2 + 24953/2852s^8t^3 - 93271/713s^7t^4 +31322/713s^6t^5 + 1016323/1426s^5t^6 - 1099301/2852s^4t^7 - 
1495957/1426s^3t^8 + 2898/31s^2t^9 + 156672/713st^{10}$. 

\smallskip 

Compute $F_2/F_3=s^2-(1527/713)st-8704/713t^2$ (Macaulay 2 \cite{m2} can do that), and find its solutions (we used again Bertini \cite{Bertini} for convenience):

\smallskip

$(s:t)=(1: 2.116e^{-1}+i 1.387e^{-17})$,

$(s:t)=(1: -3.870e^{-1} -i5.551e^{-16})$, 

\smallskip

 then apply the parameterization:
 
 \smallskip 
 
{{\begin{enumerate}
\item\label{numk}$\mathbb(f)(1: 2.116e^{-1}+i 2.220e^{-16})
=\\(1.009 -i2.457e^{-15}:
0.121-i6.207e^{-16}:
0.0234-i1.402e^{-16})$.

\item\label{numj}$\mathbb(f)(1: -3.870e^{-1} -i5.551e^{-17})
=\\(9.048:
1.086: 
0.210) $.

 \end{enumerate}}}\end{enumerate}
So we found that cases \eqref{numj} and \eqref{numk} lead (approximatively) to the same projective point: $(1:0.12:0.023)$. Those are the coordinates of the double point of $C$ (here the approximation given by Bertini is less precise).
\end{ex}

 \begin{ex}\label{3.22} Consider the quartic curve $C\subset \mathbb{P}^2$  given by the following parameterization:
$$\left\{ \begin{array}{l}
x=s^4\\
y=-s^3t+st^3\\
z=t^4
\end{array}\right. .$$  
It is easy to check that in this case $X_3$ is empty and $X_2$ is reduced and constituted of 3 points, none of them on $\mathcal{C}_2$; hence $C$ has only 3 ordinary nodes (it is easy to get that  $X_2 = \left\{(1:-\sqrt{2}:1), (1:\sqrt{2}:1), (1:0:-1)\right\}$).

By using Algorithm \ref{Algorithm2True} up to Step  \ref{Alg2Step52}, one gets that the ideal of the singular points of $C$ is $I = (w_1^2+w_0w_2-w_2^2,w_0w_1+w_1w_2,w_0^2-w_2^2)$, hence $\Sing(C) = \left\{(1:0:1), (1:-\sqrt{2}:-1), (1:\sqrt{2}:-1)\right\}$.

If one uses Algorithm \ref{Algorithm3True}, one gets that the 6 points in $\PP^1$ that are mapped to the three nodes by the parameterization of $C$ are the solutions of $F= -s^6 + s^4t^2 - s^2t^4 + t^6 = 0$, i.e. the points 
$$ (1:1), (1:-1), (2:\sqrt{2}+i\sqrt{2}), (2:-\sqrt{2}+i\sqrt{2}), (2:\sqrt{2}-i\sqrt{2}), (2:-\sqrt{2}-i\sqrt{2}).$$  
Their images via the parameterization of $C$ yield the three nodes.
\end{ex}

\begin{ex}\label{3.23} Consider the following quartic curve $C\subset\mathbb{P}^2$:
$$\left\{ \begin{array}{l}
x=s^4+s^3t\\
y=s^2t^2\\
z=st^3+t^4
\end{array}
\right..
$$
The scheme $X_3$  is empty, therefore the singularities of $C$ are all of multiplicity 2. From the matrix $M_2$ one gets the ideal of $X_2$, and it is easy to see that it is made of three simple points: $X_2=\{(1:1:1),(1:0:0),(0:0:1)\}$. We have that $(1:0:0),(0:0:1)\in \mathcal{C}_2$, hence $C$ has 2 cusps and an ordinary double point.

From Algorithm \ref{Algorithm2True}, one gets that the ideal of $\Sing(C)$ is $(w_1^2 + w_1w_2, w_0w_1 - w_1w_2, w_0w_2 + w_1w_2)$, hence the three singular points are $(1:0:0),(0:0:1),(1:-1:1)$. 

From Algorithm \ref{Algorithm3True}, one gets that those three points come (via ${\mathbf f}$) from the points in $\PP^1$ which are the zeroes of $F=s^4t^2+s^3t^3+s^2t^4$, i.e. from the points $B_1=(1:0)$, $B_2=(0:1)$, both counted twice, and $B_3,B_4 = (2:-1\pm i\sqrt3)$. It is easy to check that $B_1$ and $B_2$ give, respectively, the cusps $(1:0:0)$ and $(0:0:1)$, while $B_3$ and $B_4$ give the node $(1:-1:1)$.
\end{ex}

\begin{ex} This is   \cite[Example 1]{refCWL}, where it is studied with different tools. Consider the following quartic curve $C\subset\mathbb{P}^2$:
$$\left\{ \begin{array}{l}
x=s^4-40s^3t+40st^3+t^4\\
y=s^4+480s^2t^2+t^4\\
z=s^4+40s^3t+480s^2t^2+40st^3+t^4
\end{array}
\right..
$$
The scheme $X_3$  is empty, therefore the singularities of $C$ are all of multiplicity 2. From the matrix $M_2$ one gets the ideal of $X_2$ and it is easy to see that it is made of three simple points, none of them on $\mathcal{C}_2$, hence $\Sing(C)$ is given by three ordinary nodes.
From Algorithm \ref{Algorithm3True}, one gets that those three nodes come (via ${\mathbf f}$) from the points in $\PP^1$ which are the zeroes of $F=-st(s^4 - 2868/37s^3t - 16656/37s^2t^2 + 2868/37st^3 + t^4)$. Solving $F=0$, with 52 digit approximation with Bertini \cite{Bertini} (but we write down less decimal digits), one gets the points \\
$B_1=(1:0)$, \\$B_2=(0:1)$, \\$B_3 =(-5.583+i2.664e^{-15}:1)$, \\$B_4 = (-0.012+i   3.729e^{-17}:1)$, \\$B_5 =(0.179  -i6.938e^{-18}:1)$,\\ $B_6 = (82.930 -i3.410e^{-13}:1)$\\ and it is very easy to see that those solutions coincide with those found in \cite{refCWL}. The images of these six points give the three nodes: $B_1$ and $B_2$ give $(1,-1,1)$, while, applying the parametrization to the others, one gets  
{{
$$\mathbf{f}(B_3)=(7714.25 - i1.171e^{-11}:15939.8 - i1.613e^{-11}:8752.08 - i6.062e^{-12})=$$
$$= (0.881: 1.821:1)\sim$$
$$\mathbf{f}(B_4)=(0.517 + i1.491e^{-15}:1.069 -i 4.317e^{-16}:0.587 +i 1.060e^{-15})=$$
$$=(0.881:1.821:1),$$

$$\mathbf{f}(B_5)=( 7.934 - i2.510e^{-16}:16.395 - i1.193e^{-15}:23.788 - i1.497e^{-15})=$$
$$=(0.333:0.689:1)\sim$$
$$\mathbf{f}(B_6)=(24488700 - i4.966e^{-7}:50600600 -i 8.052e^{-7}:73417900 - i0.000)=$$
$$=(0.333:0.689:1),$$}}
 modulo approximation, the first two correspond to the same node, as well as the last two.

As noticed in \cite{refCWL}, this shows that approximation (here in Steps \ref{Alg3Step7}--\ref{Alg3Step8}  of Algorithm \ref{Algorithm3True}) challenges one to
recognize when different values appear only due to the approximation and they have to be considered as the same point. This could give problems if there are singular points which are very close to each other.  Consider that we have also the other tools ($X_3$, the intersection with $\mathcal{C}_2$) that give information a priori on the singular points.
\end{ex}

Studying the structure of the schemes $X_2,X_3,\ldots,X_{n-1}$ one should in principle be able to detect everything about the structure of the singularities of $C$. We will analyze the case of double points.

\section{Study of double points.}
 
We have that potentially, the structure of $X_2$ describes all the singularities of $C$: not only their multiplicities, but also if there are multiple tangents, infinitely near other singular points, etc.. 

Our first aim is to describe the structure of the double points of $C$; notice that with ``~double point of $C$~" we always mean  ``~singular point of $C$ of multiplicity 2~".

The possibilities for a double point $P\in C$ are many; they are classified by how many other double points there are infinitely near to $P$. Nodes need only one blowup at $P$ to give a smooth strict transform above that point, which has two points over $P$; similarly the simplest cusp $P$ needs only one blowup at $P$ to give a  strict transform which has one smooth point over $P$.

\begin{notation}
 We will use the terminology  {\it  singularity} $A_{2m-1}$, or  {\it  singularity } $A_{2m}$, $m\geq 1$, to indicate a double point which is two-branched or, respectively, a cusp (unibranched) and which needs $m$ successive blow-ups before the singularity disappears. In classical terms, $A_1$ represent an ordinary node, $A_2$ a simple cusp and $A_3$ a tacnode.  

No other possibilities are given for points of multiplicity 2; and an $A_{2m-1}$ or an $A_{2m}$ contributes as $m$ distinct double points in Clebsch formula for $C$. 
\end{notation}

 When we consider the simplest types of double points the situation is easy, and we get:
 
 \begin{rem}\label{simplydouble}
Assume $C$ has only double points and no infinitely near singular points. Every singular point gives a simple point of $X_2$. In particular: 
\begin{enumerate} 
\item\label{i} an ordinary $A_1$ double point yields a point $R\in X_2\setminus \mathcal{C}_2$ which corresponds to the secant line of $C_n$ which, in the projection $C_n \rightarrow C$ from $\Pi$, contracts two points of $C_n$ to a node on $C$.
\item\label{ii} if a double point is a cusp $A_2$ on $C$, it yields a point of $X_2$ lying on the conic {$\mathcal{C}_2 : \{4x_0x_2-x_1^2=0\}$}, corresponding to the point on $C_n$ whose tangent contracts to the cuspidal point of $C$ in the projection.
\end{enumerate}
\end{rem}

Let us notice that $\length(X_2)={n-1 \choose 2}$, and each double point of $C$ corresponds to a different point of $X_2$, hence $X_2$ is reduced and the first statement of the remark follows.

\bigskip

It is quite a classical approach in studying singularities of rational plane curves via projection from a rational normal curve $C_n$, to consider osculating spaces to $C_n$ in order to describe the structure $A_s$ of the double points (e.g. see \cite{MOE}); the following result is common knowledge:

\begin{lem}\label{Projdoublepoints} 

 Let $P\in C\subset \PP^2$ be a double point on a rational curve of degree n, then:
 \begin{enumerate}
\item \label{nodo}The point $P$ is an $A_{2m-1}$, where $2m-1<n$, if and only if the scheme $\pi^{-1}(P)$ is given by two distinct points $Q_1,Q_2\in C_n$  and $m$ is the maximum value for which  $\pi(O^{m-1}_{Q_1}(C_n))=\pi(O^{m-1}_{Q_2}(C_n))$ and the latter sets do not fill the whole $\PP^2$ (here $O^{m-1}_{Q_1}(C_n)$ is the $(m-1)$-dimensional osculating space to $C_n$ at $Q_1$).
 
\item\label{cuspfacile} The point  $P$ is an $A_{2m}$, where $m < n-1$, if and only if the scheme $\pi^{-1}(P)$ is given by the divisor  $2Q\in C_n$  and $m$ is the maximum value for which  $\pi(O^{m+1}_{Q}(C_n))$ and the latter sets do not fill the whole $\PP^2$. 
\end{enumerate}
\noindent When $m\geq 2$, we will have that the image $\pi(O^{m-1}_{Q_1}(C_n))=\pi(O^{m-1}_{Q_2}(C_n))$, (respectively $\pi(O^{m+1}_{Q}(C_n))$) is $T_P(C)$, the unique tangent line to $C$ at $P$.

 \end{lem} 

\medskip

Notice that in case \eqref{nodo} we must have that $H = \langle O^{m-1}_{Q_1}(C_n),O^{m-1}_{Q_2}(C_n)\rangle \cong \PP^{2m-1}$ is such that $\pi(H)= T_P(C)$, hence necessarily $2m-1<n$, while we need $m+1<n$ in case \eqref{cuspfacile}, otherwise $O^{m+1}_{Q}(C_n)=\PP^n$, so $\pi(O^{m+1}_{Q}(C_n)) = \PP^2$.   But this means that the description offered by Lemma  \ref{Projdoublepoints} cannot cover many cases, e.g. when $n=4$, a quartic curve $C\subset \PP^2$ can have an $A_5$ or $A_6$ singularity, but in this case the lemma does not apply. We have not been able to find in the literature a description of these situations done by projection from $C_n$; our Theorem \ref{2mtom} below will do exactly this.

 \medskip
 
If $P$ is a double point on a plane rational curve $C$, its structure $A_s$ is determined only by the number $m$ of the blowups needed to resolve the singularity and by the structure of the last blowup which yields two points ($A_s = A_{2m-1}$) or one ($A_s=A_{2m}$) points over $P$.

\bigskip

The following theorem gives an alternative way to determine $A_s$, by looking at the projection $\pi: C_n\rightarrow C$ and at the image of an appropriate length $2m$ subscheme of $C_n$.  

Recall that a 0-dimensional scheme $X\subset \PP^n$ is called {\it curvilinear}, if it is contained in a reduced curve $\Gamma$, smooth at $P =  \Supp(X)$.

In the sequel, if we have a simple point $R$ on a curve $D \subset \PP^n$,  we will denote by $sR|_D$ the curvilinear scheme in $D$ supported at $R$ and of length $s$; this is the scheme intersection of the curve with the $(s-1)$-infinitesimal neighborhood of $R$ in $\PP^n$, i.e. the scheme associated to the ideal $I_D +(I_R)^s$. 
\bigskip

\begin{thm}\label{2mtom} Let $P\in C \subset \mathbb{P}^2$, where $C$ is a rational curve and $P$ is a double point for $C$. Then:
 \begin{enumerate}
\item\label{a} $P$ is an $A_{2m-1}$ singularity if and only if 
\begin{enumerate}
 \item\label{ai}  $\pi^{-1}(P) = Q_1\cup Q_2 \subset C_n$,
 \item\label{aii}  $\pi (mQ_1|_{C_n}\cup mQ_2|_{C_n})= X $, where $X$ is a curvilinear scheme of length $m$ contained in $C$, 
 \item\label{aiii}  $m$ is the maximum integer for which \eqref{aii} holds.
\end{enumerate}

\item\label{b}  $P$ is an $A_{2m}$ singularity if and only if 
\begin{enumerate}
\item\label{bi} $\pi^{-1}(P) = Q \in  C_n$, 
\item\label{bii} $\pi (2mQ|_{C_n}) = X $, where $X$ is a curvilinear scheme of length $m$ contained in $C$,
\item\label{biii}  $m$ is the maximum integer for which \eqref{bii} holds. 
\end{enumerate}
\end{enumerate}
\end{thm} 

\begin{proof} 

It is known that a {\it normal form} for an $A_s$ singularity is given by the curve $\Gamma_s \subset \PP^2$ of equation: $y^2z^{s-1}-x^{s+1}$ at the point $p = (0:0:1)$ (e.g. see \cite{refHa, refKP}), i.e. if $C$ has an $A_s$ singularity at $P$, with $s=2m$, respectively $s= 2m-1$, then it needs $m$ successive blowups in order to be resolved and there is one, respectively two, points over $P$, and the same holds for $\Gamma_s$ at $p$. Moreover $C$  is also  analytically isomorphic to $\Gamma_s$ at $p$, and the same holds for their successive blowups. Hence we will work on these curves first. 

Case \eqref{a}: Let $(C,P)$ be an $A_{2m-1}$ singularity (hence \eqref{ai} holds). We have that $\Gamma_{2m-1}$, in affine coordinates, is the union of the two rational curves $\Gamma$: $\{y-x^m=0\}$ and $\Gamma '$: $\{y+x^m=0\}$, while $p=(0,0)$. The normalization of $\Gamma \cup \Gamma '$ is  $\psi :\tilde{\Gamma} \cup \tilde{\Gamma '}\rightarrow \Gamma \cup \Gamma '$, where $\Gamma \cup \Gamma '$ is the union of two disjoint lines and the inverse image of $p$ is given by two points, $q\in \tilde{\Gamma}$ and  $q'\in \tilde{\Gamma '}$. Since $\psi |_{\tilde{\Gamma}}: \tilde{\Gamma} \to \Gamma $ and $\psi |_{\tilde{\Gamma'}}: \tilde{\Gamma}' \to \Gamma' $ are two isomorphisms, it is clear that  $\psi(mq|_{\tilde{\Gamma}}) = mp|_{\Gamma}$ and $\psi(mq'|_{\tilde{\Gamma '}}) = mp|_{\Gamma '}$.

Now let us notice that $\Gamma \cap \Gamma'$ is the curvilinear scheme $Z$ whose ideal is $(y,x^m)$, hence $mp|_{\Gamma} = mp|_{\Gamma'} = Z$, and $\eqref{aii}$ is true.  On the other hand, if we consider $\psi((m+1)q|_{\tilde{\Gamma}} \cup (m+1)q'|_{\tilde{\Gamma '}})$, we get the scheme $(m+1)p|_{\Gamma} \cup (m+1)p|_{\Gamma '}$ which corresponds to the ideal $(x^{m+1},xy,y^2)$, so such scheme is not curvilinear and \eqref{aiii} holds.

Case \eqref{b}: Let $(C,P)$ be an $A_{2m}$ singularity (hence \eqref{bi} holds). We have that $\Gamma_{2m}$, in affine coordinates, is the irreducible curve $y^2-x^{2m+1}$, hence a parameterization (which is also a normalization) for it is $\phi : \mathbb{A}^1 \to \Gamma_{2m}$, where $\phi(t) = (t^2,t^{2m+1})$ and $p=(0,0)$ is such that $\phi^{-1}(p)=q$.  The ring map corresponding to $\phi$ is:  
$$
\tilde{\phi} : \frac{K[x,y]}{(y^2-x^{2m+1})} \to K[t], \quad \overline{x} \mapsto t^2, \quad \overline{y} \mapsto t^{2m+1}.
$$ 

The scheme $2mq|_{\mathbb{A}^1}$ corresponds to the ideal $(t^{2m})$, and $\tilde \phi^{-1} ((t^{2m}))= (y,x^m) $, hence \eqref{bii} is true. 

On the other hand, the scheme $2(m+1)q|_{\mathbb{A}^1}$ corresponds to the ideal $(t^{2m+2})$, and $\phi^{-1} ((t^{2m+2}))= (y^2,xy,x^{m+1}) $, hence \eqref{bii} is true. 

To check that the ``~if~'' part of statements \eqref{a} and \eqref{b} holds, just consider that the \eqref{ai} and the \eqref{bi} parts determine if the singularity is of type $A_{2h}$ or $A_{2h-1}$, while the structure of $\Gamma _s$ forces $h=m$. 

Now let us notice that, being $\Gamma_s$  at $p$ analytically isomorphic to $C$ at $P$, when we consider $\Gamma_s$ and $C$ as analytic complex spaces, there exist open Euclidean neighborhoods $U$ of $p$ and $V$ of $P$ such that $U\cap \Gamma_s$ and $V\cap C$ are biholomorphically equivalent (e.g. see  \cite[Corollary 2.10, Ch.6]{BS}). Since the statement is of local nature, this is enough to conclude.
 \end{proof}

 \medskip
The following proposition relates the structure of $X_2$ at a point $R\in X_2$ to the value of the invariant $\delta_P$, for the double point $P\in \Sing(C)$ associated to the secant (or tangent) line parameterized by $R$.

\begin{prop}\label{SevVar} Let $C,{\bf f}, X_2$ be as in Section \ref{Preliminaries}. Let $C$ have only double points as singularities and let $R\in X_2$ and $P\in \Sing(C)$ be the point associated to $R$. Then $\length_R(X_2) = \delta_P$. 
\end{prop}

\begin{proof} It is well known that every rational plane curve $C_0$ can be viewed as the limit of a family of plane rational nodal curves $C_t$; moreover every singular point $P\in C_0$ is the limit for $t\rightarrow 0$ of $\delta_P$  nodes of $C_t$ (see \cite[Theorem 2.59,(1),(c)]{gls}). Hence by considering the associated family of the schemes $X_{2,t}$ we have that its generic element is given by simple points and at each point $R$ of $X_{2,0}$ we have that the structure of $X_{2,0}$ at $R$ is the limit of $\delta_P$ simple points where $P\in \Sing(C_0)$ is the point associated to the secant (or tangent) line parameterized by $R$. Therefore $\length_R(X_2) = \delta_P$. 
\end{proof}

This gives that the complete configuration of the double points of $C$ is visible in the structure of $X_2$:

 \begin{cor}\label{doublepts}
Let $C$ have only double points as singularities. Let $P\in \Sing(C)$. Then, from Proposition \ref{SevVar}, we have: 
\begin{enumerate}
\item $P$ is an $A_{2m-1}$ singularity  iff $P=\pi(L_R)$ where $R\in X_2\setminus \mathcal{C}_2$ corresponds to a secant line $L_R$ of $C_n$ and  $\length_R(X_2)=m$;

\item  $P$ is an $A_{2m}$ singularity  iff  $P=\pi(L_R)$ where $R\in X_2\cap \mathcal{C}_2$ corresponds to a tangent line $L_R$ of $C_n$ and $\length_R(X_2)=m$.
\end{enumerate}
\end{cor}

The following examples show how one can analyze the presence of $A_{2m-1}$ and $A_{2m}$ singularities from the structure of $X_2$, by using Corollary \ref{doublepts}.

\begin{ex}
Consider the following quartic $C\subset\mathbb{P}^2$:
$$\left\{ \begin{array}{l}
x=s^3t-3st^3\\
y=s^2t^2\\
z=2s^4-8s^2t^2+9t^4
\end{array}
\right..
$$

Here $X_3$ is empty and $\length(X_2)=3$ while $\length (\Supp(X_2))=2$, and  none of the points of $X_2$ lies on $\mathcal{C}_2$. Hence, by  Corollary \ref{doublepts}, $C$ has one ordinary double point and  an $A_3$ singularity (a tacnode).
\end{ex}

\begin{ex}
Consider the following septic curve $C\subset\mathbb{P}^2$:
$$\left\{ \begin{array}{l}
x=s^2t^5\\
y=s^7\\
z=s^7+s^6t+s^3t^4+st^6+t^7
\end{array}
\right..
$$
The scheme $X_3$ is empty. The scheme $X_2$ has support at 13 points and has length 15. Let $I_0 = I_{X_2}$, and $I_m = I_{m-1}: I_{\mathcal{C}_2}$. It is not hard to check that $I_m$ defines a scheme of length $15-m$, for $m=0,1,2,3$, and that $I_m = I_3$ for $m\geq 4$.  Hence $X_2$ is made of 12 simple points outside $\mathcal{C}_2$ and at another point, $R\in \mathcal{C}_2$, it has $\length_R(X_2)=3$. So $C$, by  Corollary \ref{doublepts}, has an $A_6$ cusp and 12 ordinary nodes.  
\end{ex}

\begin{ex}\label{cuspidal5ic} Consider the following quintic $C\subset\mathbb{P}^2$:
$$\left\{ \begin{array}{l}
x=s^2t^3\\
y=s^5\\
z=s^5+s^4t+t^5
\end{array}
\right..
$$
The scheme $X_3$ is empty. The scheme $X_2$ has support at 5 points and length 6. The point $R=(1:0:0) \in X_2$ is the only one on $\mathcal{C}_2$ and it has has $\length_R(X_2)=2$. Therefore $C$ has 4 nodes $A_1$ and an $A_4$ cusp, by  Corollary \ref{doublepts}.
\end{ex}

\bigskip

Sometimes it is difficult to find the value of $\length_R(X_2)$ for each $R\in X_2$. Moreover we are also interested in determining the coordinates of the singular points of $C$ in the plane. Therefore we will define a procedure (similar to what we did in Algorithm \ref{Algorithm3True}) which might allow us to determine the structure of the double points of a given $C$ by using Proposition \ref{2mtom}, and also get their coordinates from their preimages on $\PP^1$. Such algorithm would surely give an answer if the following conjecture (suggested by many examples) were true:

\begin{conjecture}\label{multiplelines} Let $C,{\bf f}, X_2$ be as in Section \ref{Preliminaries}. Let $C$ have only double points as singularities, let $\Sing(C)=\{P_1,\dots,P_t\}$ and $(X_2)_1,\dots, (X_2)_t$ the corresponding subschemes of $X_2$. Then
\begin{itemize}
\item $(X_2)_i$ is curvilinear and $\length (X_2)_i=\delta _{P_i}$;
\item if $P_i$ is an $A_{2m_i -1}$ singularity, then $L_i:=p_2p_1^{-1}((X_2)_i)$ cuts two curvilinear schemes of length $m_i$ on $C_n$, such that their image under $\pi$ is a curvilinear subscheme of length $m_i$ of $C$;
\item if $P_i$ is an $A_{2m_i}$ singularity, then $L_i:=p_2p_1^{-1}((X_2)_i)$ is tangent to $C_n$ and cuts a curvilinear scheme of length $2m_i$ on $C_n$, such that its image under $\pi$ is a curvilinear subscheme of length $m_i$ of $C$. 
\end{itemize}
\end{conjecture}
  
In other words, the above conjecture states that if there is an $A_s$ singularity, its description as in Theorem \ref{2mtom} is given by the scheme $p_2p_1^{-1}(X_2) \cap C_n$ and its projection via $\pi$ on $C$.
 \bigskip
 
 \begin{prop}\label{doubleptsproj}
Let $C,{\mathbf f}$ be as above, and such that $C$ has only double points as singularities. Then its singularities are determined by the following algorithm,
provided that Conjecture 4.9 holds: 
\end{prop}

\begin{algorithm}[H]\caption{Coordinates of Singularities of $C$}\label{Algorithm4}

 \textbf{Input}: $M_i\in \mathbb{C}^{(n-k+4)\times (n+1)}$, $i=2, \ldots , n-1$ (see \eqref{trueMk}).\\
\textbf{Output}: $\Sing(C)$.

\begin{boxedminipage}{125mm}

    \begin{algorithmic}[1]

\STATE\label{Alg4Step1} Same as Step \ref{Alg2Step1} of Algorithm \ref{Algorithm2True}.
\STATE Run Algorithm \ref{Algorithm2True} from Step \ref{Alg2Step13} to Step \ref{Alg2Step3}  but starting with $I_{X_2}\subset K[x_0,x_1,x_2]$.  
Get the ideal $I_L\subset K[z_0,\ldots,z_n]$ which defines a scheme $L$ in $\PP^n$ supported on the lines $L_1,\ldots,L_h$ ($h=$  cardinality of $\Supp X_2=$  cardinality of $\Sing(C)$), parameterized by $X_2$. These lines are exactly the secant or tangent lines to $C_n$ which get contracted via $\pi$ to a singular point $P\in C$. The ideal $I_L$ gives a scheme supported at these lines which has degree equal to $\length_{R_j}X_2$ at $L_j$, if $R_j$ parametrizes $L_j$;
\STATE\label{Alg4Step3} Run Algorithm \ref{Algorithm3True} from Step \ref{Alg2Step2}
to Step \ref{Alg3Step7} using $I_L$ (instead of $I_{L_k}$ of Step \ref{Alg3Step4} in Algorithm \ref{Algorithm3True}). Get, by approximation,  the  zeroes of $F$ (instead of $F_k$ of Step \ref{Alg3Step7} in Algorithm \ref{Algorithm3True}). They give  a subscheme $\Gamma$ of $\PP^1$ which is $\nu_n^{-1}(C_n\cap L)$. 
 \STATE Write $\Gamma$ as a divisor.

 \IF {$\Gamma$  is  of type 
 \begin{equation}\label{Alg4Step5} m_1B_{1}+m_{1}B'_{1} + \cdots + m_{c}B_c + m_cB'_c + 2m_{c+1}\mathbf{B}_{c+1} + \cdots + 2m_{h-c}\mathbf{B}_{h-c}\end{equation} where $0\leq c\leq h$ and $(B_j,B'_j)$ correspond to a secant $L_j$, while each $\mathbf{B}_j$ to a tangent $L_j$ and $\sum m_j = {n-1 \choose 2}$,}
 
 \STATE{go to Step \ref{Alg4old8},}
 \ELSE \STATE {The algorithm stops here with no answer.}
\ENDIF

\STATE\label{Alg4old8} Run Algorithm \ref{Algorithm2True} from Step \ref{Algo2Step8} to Step \ref{Alg2Step5} starting with $I_S$ (instead of $I_L$ of Step \ref{Algo2Step8} in Algorithm \ref{Algorithm2True}). Get 
 the elimination ideal $
 I_N\subset K[w_0,w_1,w_2]$ 
which is the  projection $S'$ of $S$ into $\PP^2$.
  \STATE  Compute $\Supp(S')=\{P_1,\ldots ,P_h\}$. This is  $\Sing(C)$;
\IF {$P_j$ corresponds either to $\mathbf{B}_j$ or to $(B_j,B'_j)$ in \eqref{Alg4Step5}, and $S'$ is curvilinear of length $m_j$ at $P_j$}
\STATE {$P_j$ is either an $A_{2m_j}$ or, respectively, an $A_{2m_j-1}$ singularity and go to Step \ref{Alg4Last};}
\ELSE \STATE {The algorithm stops without an answer;}
\ENDIF
 \STATE\label{Alg4Last} By applying ${\mathbf{f}}$ either to the $\mathbf{B}_j$'s or to the $(B_j,B'_j)$'s one gets the (maybe approximated) coordinates of points $P_1,\ldots,P_h$ in $\Sing(C)$.

    \end{algorithmic}
\end{boxedminipage}

\end{algorithm}

\begin{proof}
Algorithm \ref{Algorithm4} works in the same way as Algorithm \ref{Algorithm3True} does, but it uses also the non-reduced structure of $X_2$, following the idea of Conjecture \ref{multiplelines}. If the algorithm can run to the end, it works by Proposition \ref{2mtom}. If Conjecture \ref{multiplelines} is true, the algorithm will always work, otherwise there could be cases where it just stops without results. 
\end{proof}

  \bigskip

 \begin{ex}\label{quintic} Consider the quintic curve in Example \ref{cuspidal5ic}, $C\subset\mathbb{P}^2$:
$$\left\{ \begin{array}{l}
x=s^2t^3\\
y=s^5\\
z=s^5+s^4t+t^5
\end{array}
\right..
$$
We have seen that,  by Proposition \ref{SevVar}, $C$ has 4 nodes $A_1$ and an $A_4$ cusp. Let us check that it is really so using Algorithm \ref{Algorithm4}.  If we run Steps \ref{Alg4Step1}  to \ref{Alg4Step3}, we get the polynomial $F=-s^{12} - s^8t^4 – s^4t^8$, whose zeros, as a divisor in $\PP^1$, give: $B_1+B_1'+B_2+B_2'+B_3+B_3'+B_4+B_4'+4\mathbf{B}_5,$, where: 
$$B_1=(e^{i\frac{\pi}{6}}:1),   B_1'=(e^{i\frac{5\pi}{6}}:1); B_2=(e^{i\frac{2\pi}{3}}:1), B_2'=(e^{-i\frac{2\pi}{3}}:1);$$ $$B_3=(e^{i\frac{7\pi}{6}}:1), B_3'=(e^{-i\frac{\pi}{6}}:1); B_4=(e^{-i\frac{\pi}{3}}:1), B_4'=(e^{i\frac{\pi}{3}}:1)$$ 
which, by applying $\mathbf{f}$ yield, respectively, the four nodes in $\PP^2$: $P_1=(1:i:1+i); P_2=(1:1:0); P_3=(2:1+i\sqrt{3}:2); P_4=(1:-1:2)$. Then we have $\mathbf{B}_5=(0:1)$, and $P_5 = \mathbf{f}(\mathbf{B}_5)=(0:0:1)$. 
Running Steps \ref{Alg4old8}  to \ref{Alg4Last} of Algorithm \ref{Algorithm4}, one gets the ideal
 
$(-x_0^2x_2 - x_0x_1x_2 - 2x_1^2x_2 + x_1x_2^2,\ -x_0^2x_1 + x_1^3 - 2x_1^2x_2 + x_1x_2^2,$ 

$\qquad \qquad \qquad \qquad \qquad \qquad -x_0^3 + x_1^3 - x_1^2x_2, x_0x_2^2 - x_1^3 - x_0x_1x_2,\ 4x_1^3x_2 + x_0x_1x_2^2 - 3x_1^2x_2^2 + x_1x_2^3)$

\noindent which defines a scheme given by the simple points $P_1,\ldots ,P_4$ and is curvilinear of lenght 2 at $P_5$, thus showing that $C$ has a cusp $A_4$ at $P_5$.
 
\end{ex}

\section{The real case: hidden singularities.}

 When $f_0,f_1,f_2$ are forms in $\RR[s,t]_n$, let us consider the image of the parameterization, $C_\RR \subset \PP_\RR^2$. We want to check whether our method of studying singularities of $C$ still gives results when working on real rational curves. Real rational curves are very much studied in applications, for general reference see e.g. \cite{refSWP}.   Here when we use the notation of the previous section: $C,\ C_n,\ \Pi,\ X_2$, etc. we will assume that we are working on $\CC$.

As above, we will assume that  $(f_0:f_1:f_2)$ is a proper parameterization, i.e. that it is generically one to one and the $f_i$'s do not have non-constant common factors. 

The whole procedure to construct the schemes $X_2,X_3,\ldots ,X_{n-1}$ works also over $\RR$, but in this case not all the points of the space $\PP^k_\RR$ do parameterize $(k-1)$-spaces $k$-secant to $C_{n,\RR}\subset \PP^n_\RR$, but only the ones which correspond to forms in $\RR[s,t]_n$ which are a product of linear factors (they form a semi-algebraic set, defined by equations and inequalities in $\PP^k_\RR$). Here we will consider the real rational normal curve $C_{n,\RR}\subset \PP_\RR^n$ which is defined by $\nu_n:\PP^1_\RR \rightarrow \PP^n_\RR$ as before. Notice that the real secant variety $\sigma_2(C_{n,\RR})$, when it does not fill $\PP^n_\RR$, is no longer an algebraic variety, but just a semi-algebraic set. 

This is not the only problem in working with the real case. We have that $C_\RR$ itself might not be an algebraic set. Once we find an implicit equation $F=0$ for $C_\RR$, there may be points satisfying such equation which are not in $C_\RR$, i.e. there may be isolated points (also called {\it acnodes}) of the curve, which are not contained in the part parameterized by $(f_0:f_1:f_2)$ (so acnodes will be points with real coordinates coming from non-real points in the parameterization $ \mathbf{f} : \PP^1_\CC \rightarrow \PP^2_\CC$). Another possibility is that an infinitesimal neighborhood of some point of $C_\RR$ is missing; let us see this more in detail.

First observe that $F$ is irreducible over $\RR$; in fact, let $C$ be the complex curve given by the parameterization of $C_\RR$ read on the complex field $\CC$; $C$ is rational, hence integral, so the polynomial $F$ (which is the implicit equation for $C$ on $\CC$) is irreducible over $\CC$ and hence is irreducible over $\RR$, too. Now denote by $\overline C_\RR$ the curve in $\PP^2_\RR$ associated to the ideal $(F)$; $F$ being irreducible, this is an integral curve on $\RR$. Sometimes  $\overline C_\RR =  C_\RR$, but sometimes they are different; in  example \ref{oneacnode}  the two sets differs for the point $P$, which is an acnode; in example \ref{realtriple} they are not different as sets, but $\overline C_\RR$ contains the second infinitesimal neighborhood of the point $P$, while $ C_\RR$, considered as a differential curve, is smooth. 

\par In both examples $P$ is a singular point for $C$ (with complex tangents), hence we will consider it as a singularity for $\overline C_\RR$, too. Such points are not immediately detected via the parameterization, hence our schemes $X_k$ could be a good way to find them.

The following examples may clarify what can happen in the real case.

\begin{ex}\label{oneacnode} 
Let $C_\RR$ be the cubic defined by the parameterization:
$$\left\{ \begin{array}{l}
x=t^3+s^2t\\
y=-s^3-st^2\\
z=-t^3
\end{array}
\right..
$$

\noindent we have that an implicit equation is $F=x^3+x^2z+y^2z=0$; and that $P=(0:0:1)$ is a singular point for $F=0$, but $P\notin C_\RR$, i.e. $\overline C_\RR = C_\RR\cup P$. If we compute $X_2 = X_{2,\RR}$, we find that it is given by the reduced point $(1:0:1)\in \PP^2_\mathbb{R}$, which corresponds to the form $s^2+t^2\in \RR[s,t]_2$, which is irreducible, hence it does not give a secant line of $C_{3,\RR}$. By using equation \eqref{eqn1}: $\sum_{i=0}^k x_iz_{i+j} = 0, \; \; j=0, \ldots , n-k$, we get that the corresponding line $L_P\subset \PP^3_\RR$ is given by $\{z_0+z_2=0; z_1+z_3=0\}$, which intersects $C_3$ (in $\PP^3_\CC$) in the points: $Q_1=(1:i,-1:-i)$ and $Q_2=(1:-i,-1:i)$; and we have that $\pi(Q_1)=\pi(Q_2)=P$.
Actually, by running Algorithm \ref{Algorithm3True}, one finds the points $(1:i)$ and $(-1:i)$ in $\PP^1_\mathbb{C}$, whose images in $C$ via $\mathbf{f}$ coincide and give the double point (acnode) $P$. 

\end{ex}

\medskip

The following example shows another kind of ``~hidden~'' singularity in a real curve; here a triple point appears as a simple one, since two of its branches are complex.

\begin{ex}\label{realtriple} 
Let $C_\RR$ be the quartic curve defined by the parameterization:
$$\left\{ \begin{array}{l}
x=s^4+s^2t^2\\
y=-s^3t-st^3\\
z=-s^4-t^4
\end{array}
\right..
$$

 If we compute $X_3 = X_{3,\RR}$, we find that it is given by the reduced point $R=(1:0:1:0)\in \PP^3_\mathbb{R}$, which corresponds to the form $s^3+st^2=s(s^2+t^2)\in \RR[s,t]_3$, whose second factor is irreducible, hence $H_R$ will be a trisecant plane of $C_4\subset \PP^4_\mathbb{C}$, but its real part will only give a plane which intersects $C_{4,\RR}$ in one point. The scheme $X_2\subset  \PP^2_\mathbb{C}$ is made of three simple points (corresponding to the three secant lines in $H_R$) and only one of them, $R_1=(1:0:1)$,  has real coordinates, and it corresponds to the secant line containing the two complex points of $H_R\cap C_4$. If we run Algorithm \ref{Algorithm3True}, we find the polynomial $F = s^3 + st^2$ which gives the points $(0:1)$, $(1:i)$ and $(-1:i)$ in $\PP^1_\mathbb{C}$, whose image in $C$ via $\mathbf{f} $  gives the triple point $P=(0:0:1)\in \overline{C}_\RR$. Actually, the implicit equation of $\overline{C_\RR}$ is $x^4 + y^4 + x^3z + xy^2z=0$, whose first and second  derivatives are zero at $P$ (hence $P$ is a triple point for $C$).  Considered as a curve from a differential point of view, $C_\RR$ is smooth.  

\end{ex}

 The two examples show how we can use Algorithm \ref{Algorithm3True} (or Algorithm \ref{Algorithm4}) to study the hidden singularities in the real case. The algorithm gives a form $F\in \CC[s,t]$ which gives the points in $\PP^1_\CC$ whose image in $C$ gives the singularities; by checking when more points in $\PP^1_\CC$ have the same (real) image into $\PP^2_\RR$, we can find acnodes (isolated singularities) or ``~hidden singularities~" for $C_\RR$, as in the two examples above. 
 
 Other examples showing how our algorithms work in the real case are Example \ref{quintic} and Examples \ref{3.22} and \ref{3.23}. In Example  \ref{quintic}, when we consider the curve $C_\RR$, two of the nodes we found have real coordinates, but they are acnodes, since they come (via $\mathbf{f})$ from two pairs of complex points in $\PP^1$, so $\overline{C}_\RR$ has two isolated points, and $C_\RR$ is singular only at $P_5$ (which remains an $A_4$ cusp).  In the same way,  Examples \ref{3.22} and \ref{3.23} present two acnodes and one acnode, respectively.

\bigskip
ACKNOWLEDGMENTS

We would like to thank the two anonymous referees for their kind words about the paper and for their help to make the paper more readable and to correct several math inadequacies. Moreover we like to thank Edoardo Sernesi for useful talks.


\end{document}